\documentclass[12pt]{article}

\usepackage[english]{babel}
\usepackage[utf8]{inputenc}
\usepackage{amsmath}
\usepackage{graphicx}
\usepackage[colorinlistoftodos]{todonotes}
\usepackage{polynom}
\usepackage{amsfonts}
\usepackage{dsfont}
\usepackage{amsthm}
\usepackage{hyperref}
\usepackage{graphicx}
\newtheorem{theorem}{Theorem}[section]

\newtheorem{lemma}{Lemma}[section]
\newtheorem{proposition}[theorem]{Proposition}

\usepackage{fullpage}
\usepackage[T1]{fontenc}
\usepackage{hyperref}
\usepackage{xcolor}
\definecolor{link-color}{rgb}{0.15,0.4,0.15}
\hypersetup{
  colorlinks,
  linkcolor = {link-color}, citecolor = {link-color},
}

\usepackage{graphicx}
\usepackage{caption}
\usepackage{subcaption}
\usepackage{hyperref}
\usepackage{amsmath,amsthm,amssymb}
\usepackage{bbm}
\usepackage{tabularx}

\newcommand{\RR}{\mathbb{R}}

\newcommand{\NN}{\mathbb{N}}
\newcommand{\CC}{\mathbb{C}}

\newcommand{\cdotsc}{,\dotsc,}

\newcommand{\stproca}[1]{\left(#1\right)_{t \ge 0}}
\newcommand{\stproc}[1]{\stproca{#1_t}}

\DeclareMathOperator{\rRe}{Re}

\renewcommand{\Re}{\rRe}

\newcommand{\iu}{\mathrm{i}} % imaginary unit

\newenvironment{eqnarr}{\begin{IEEEeqnarray}{rCl}}{\end{IEEEeqnarray}\ignorespacesafterend}
\renewcommand{\eqref}[1]{\hyperref[#1]{(\ref*{#1})}}

\newcommand{\rhohat}{\hat{\rho}}

\newcommand{\GGt}{(\mathcal{G}_t)_{t \ge 0}}

\DeclareMathOperator{\diag}{diag}
\newcommand{\Had}{\circ}
\DeclareMathOperator{\sgn}{sgn}

\newcommand{\dd}{\mathrm{d}}

\newcommand{\for}{\qquad}

\newcommand{\define}{\emph}

\newcommand{\abs}[1]{\lvert #1 \rvert}

    \def\beq{\begin{eqnarr}}%arw
    \def\eeq{\end{eqnarr}}%arw
    \def\beqq{\begin{eqnarray*}} %arw
    \def\eeqq{\end{eqnarray*}} %arw

\newcommand*{\pref}[1]{\hyperref[#1]{(\ref*{#1})}}
\newcommand*{\refpref}[2]{\hyperref[#2]{\ref*{#1}(\ref*{#2})}}

\title{Conditioned real self-similar Markov processes}

\author{Andreas E. Kyprianou\footnote{Department of Mathematical Sciences, University of Bath, Claverton Down, Bath, BA2 7AY, UK. Email: \texttt{a.kyprianou@bath.ac.uk}}, \ V{\'\i }ctor M. Rivero\footnote{CIMAT A. C.,
Calle Jalisco s/n,
Col. Valenciana,
A. P. 402, C.P. 36000,
Guanajuato, Gto.,
Mexico.
Email: \texttt{rivero@cimat.mx}}, \ Weerapat Satitkanitkul\footnote{Department of Mathematical Sciences, University of Bath, Claverton Down, Bath, BA2 7AY, UK. Email: \texttt{ws250@bath.ac.uk}}}

\date{\today}

\begin{document}
\maketitle
\begin{abstract}
\noindent In recent work, Chaumont et al. \cite{CPR} showed that is possible to condition a stable process with index $\alpha\in(1,2)$ to avoid the origin. Specifically, they describe a new Markov process which is the   Doob $h$-transform of a stable process and which arises from a limiting procedure in which the stable process is conditioned  to have avoided the origin at later and later times. A stable process is a particular example of a real self-similar Markov process (rssMp) and we develop the idea of such conditionings further  to the class of rssMp. Under appropriate conditions, we show that the specific case of conditioning to avoid the origin corresponds to a classical Cram\'er-Esscher-type transform to the  Markov Additive Process (MAP) that underlies the Lamperti-Kiu representation of a rssMp. In the same spirit, we show that the notion of conditioning a rssMp to continuously absorb at the origin also fits the same mathematical framework. In particular, we characterise the stable process conditioned to continuously absorb at the origin when $\alpha\in(0,1)$. Our results also complement related work for positive self-similar Markov processes in \cite{CR}.
\end{abstract}

\section{Introduction}
This work concerns conditionings of real self-similar Markov processes (rssMp) and so we start by characterising this class of stochastic processes.

A rssMp with {\it self-%
similarity index} $\alpha > 0$ is a standard Markov process $X = \stproc{X}$  (in the sense of \cite{BG-mppt})
with probability laws
$(\mathbb{P}_x)_{x \in \RR}$, which satisfies the
{\it scaling property} that for all $x \in \RR \setminus \{0\}$
and $c > 0$,
\[ \text{the law of }(c X_{t c^{-\alpha}})_{t \ge 0}
  \text{ under } \mathbb{P}_x \text{ is } \mathbb{P}_{cx} . \]

The structure of real self-similar Markov processes has been investigated
by \cite{Chy-Lam} in the symmetric case, and \cite{CPR} in general.
Here, we give an interpretation of these authors' results
in terms of Markov additive process (MAP) with a two-state modulating Markov chain and therefore we make an immediate digression to introduce such processes.

\subsection{Markov Additive Processes}\label{ss:MAP}

Let $E$ be a finite state space and $\GGt$ a standard
filtration. 
A c\`adl\`ag process $(\xi,J)$ in $\mathbb{R} \times E$
with law $\mathbf{P}$ is called a
\define{Markov additive process (MAP)} with respect to $\GGt$
if $(J(t))_{t \ge 0}$ is a continuous-time Markov chain in $E$, and
the following property is satisfied,
for any $i \in E$, $s,t \ge 0$:
\begin{eqnarray}
\label{e:MAP}
& \text{ given $\{J(t) = i\}$,
the pair $(\xi(t+s)-\xi(t), J(t+s))$ is independent of
$\mathcal{G}_t$,}\notag\\
&\text{ and has the same distribution as $(\xi(s)-\xi(0), J(s))$
given $\{J(0) = i\}$.}
\end{eqnarray}

Aspects of the theory of Markov additive processes
are covered in a number of texts, among them
\cite{Asm-rp1} and \cite{Asm-apq2}. More classical work includes \cite{Cinlar1, Cinlar2, AS} amongst others.
We will mainly use the notation of
\cite{Iva-thesis},
where it was principally assumed that 
$\xi$ is spectrally negative; the results which
we quote are valid without this hypothesis, however.

Let us introduce some notation.
For $x\in\mathbb{R}$,  write $\mathbf{P}_{x,i} = \mathbf{P}( \cdot \,\vert\, \xi(0) = x, J(0) = i)$.
If $\mu$ is a probability distribution on $E$, we write
$\mathbf{P}_{x, \mu}  
  = \sum_{i \in E} \mu_i \mathbf{P}_{x,i}$. % \mu_i?
% It will also be convenient to denote by $\mathbf{P}(A)$ the column vector
% with $i$th element $\mathbf{P}_i(A)$,
% and by $\mathbf{P}(A; J(t))$ the matrix whose $(i,j)$th entry
% is $\mathbf{P}_i(A, J(t) = j)$.
We adopt a similar convention
for expectations.

It is well-known that a Markov additive process $(\xi,J)$ also satisfies
\eqref{e:MAP} with $t$ replaced by a stopping time, albeit on the event that the stopping time is finite. The following proposition gives a characterisation of MAPs in terms of a mixture of L\'evy processes, a Markov chain and a family of additional jump distributions;
see \cite[\S XI.2a]{Asm-apq2} and \cite[Proposition 2.5]{Iva-thesis}.

\begin{proposition}
  The pair $(\xi,J)$ is a Markov additive process if and only if, for each $i,j\in E$, 
  there exist a sequence of iid L\'evy processes
  $(\xi_i^n)_{n \ge 0}$ and  a sequence of iid random variables
  $({\Delta}_{i,j}^n)_{n\ge 0}$, independent
  of the chain $J$, such that, if $\sigma_0 = 0$
  and $(\sigma_n)_{n \ge 1}$ are the
  jump times of $J$,  then the process $\xi$ has the representation
  \[ \xi(t) = \mathbf{1}_{(n > 0)}( \xi(\sigma_n -) + {\Delta}_{J(\sigma_n-), J(\sigma_n)}^n) + \xi(J(\sigma_n))^n(t-\sigma_n),
    \for t \in [\sigma_n, \sigma_{n+1}),\, n \ge 0. \]
\end{proposition}
For each $i \in E$, it will be convenient to define $\xi_i$ as a L\'evy process  whose distribution is the common
  law of the $\xi_i^n$ processes in the above representation; and similarly, for each $i,j \in E$, define ${\Delta}_{i,j}$ to
  be a random variable having the common law of the ${\Delta}_{i,j}^n$ variables.

Henceforth, we confine ourselves
to irreducible (and hence ergodic) Markov chains $J$.
Let the state space $E$ be the finite set $\{1 \cdotsc N\}$, for some $N \in \NN$.
Denote the transition rate matrix of the chain $J$ by
${\boldsymbol{Q}} = (q_{i,j})_{i,j \in E}$.
For each $i \in E$, the Laplace exponent of the L\'evy process $\xi_i$
will be written $\psi_i$. To be more precise, for all $z\in\mathbb{C}$ for which it exists, 
\[
\psi(z): = \log\int_\mathbb{R}{\rm e}^{zx}{\rm P}(\xi(1)\in{\rm d}x).
\]
For each pair of $i,j \in E$,
define
the Laplace transform $G_{i,j}(z) = \mathrm{E}[{\rm e}^{z {\Delta}_{i,j}}]$
 of the jump
distribution ${\Delta}_{i,j}$,
whenever this exists. Write ${\boldsymbol{G}}(z)$ for the $N \times N$ matrix
whose $(i,j)$-th element is $G_{i,j}(z)$. %\footnote{Here, we understand the Laplace transform  of a random variable $W$ with law $P$ (and associated expectation operator $E$) to mean $E(\exp\{z W\})$, for all $z\in\mathbf{C}$ such that the expectation is finite. Moreover, we define the associated Laplace exponent as $\log E(\exp\{z W\}).$}
We will adopt the convention that ${\Delta}_{i,j} = 0$ if
$q_{i,j} = 0$, $i \ne j$, and also set ${\Delta}_{ii} = 0$ for each $i \in E$.

The multidimensional analogue of the Laplace exponent of a L\'evy process is
provided by the matrix-valued function
\begin{equation}\label{e:MAP F}
 {\boldsymbol F}(z) = \diag( \psi_1(z) \cdotsc \psi_N(z))
  + {\boldsymbol{Q}} \Had {\boldsymbol{G}}(z),
\end{equation}
for all $z \in \CC$ such that the elements on the right are defined,
where $\Had$ indicates elementwise multiplication, also called
Hadamard multiplication.
It is then known
%\cite[Lemma 2.1]{AK-mg}
that
\[ \mathbf{E}_{0,i}( {\rm e}^{z \xi(t)} ; J(t)=j) = \bigl({\rm e}^{\boldsymbol{F}(z) t}\bigr)_{i,j} , \for i,\,j \in E,  t\geq 0,\]
such that the right-hand side of the equality is defined.
For this reason, $\boldsymbol{F}$ is called the \define{matrix exponent} of
the MAP $(\xi, J)$.

\bigskip

The role of $\boldsymbol{F}$ is analogous to the role of the Laplace exponent of a L\'evy process. Similarly in this respect,  one might also regard the {\it leading eigenvalue} associated to $\boldsymbol F$  (sometimes referred to as the 
{\it Perron--Frobenius eigenvalue},
see \cite[\S XI.2c]{Asm-apq2} and \cite[Proposition 2.12]{Iva-thesis})  as also playing this role.

\begin{proposition}
Suppose that $z \in \mathbb{R}$ is such that ${\boldsymbol F}(z)$ is defined.
Then, the matrix ${\boldsymbol F}(z)$ has a real simple eigenvalue $\chi(z)$,
which is larger than the real part of all its other eigenvalues.
Furthermore, the
corresponding
right-eigenvector ${\boldsymbol v}= (v_1(z), \cdots, v_N(z))$ may be chosen so that
$v^z_i > 0$ for every $i = 1,\cdots N$,
and normalised such that
\begin{eqnarray}
  {\boldsymbol \pi} \cdot {\boldsymbol v}(z) &=& 1 \label{e:h norm}
\end{eqnarray}
where $  {\boldsymbol \pi} = (\pi_1, \cdots, \pi_N)$ is the equilibrium distribution of
the chain $J$.
\end{proposition}
%{\color{blue} VR: is this result correct for any $z\in\mathbb{C}$? My guess is that no, according to the second theorem below the eigenvalue is not always defined...double check.}
One sees the leading eigenvalue appearing in a number of key results. We give two such below that will be of pertinence later on. The first one is  the strong law of large numbers for $(\xi, J)$, in which the leading eigenvalue plays the same role as the Laplace exponent of a L\'evy process does in analogous result for that setting. The following result is taken from \cite[Proposition 2.10]{Asm-apq2}.

\begin{proposition}
If $\chi'(0)$ is well defined (either as a left or right derivative), then  we have
\begin{align}\label{mmm}
\lim_{t\to\infty}\frac{\xi(t)}{t} = \chi'(0)={\bf E}_{0,{\boldsymbol\pi}}[\xi(1)]: = \sum_{i\in E}{\boldsymbol \pi}_i {\bf E}_{0,i}[\xi(1)]
\end{align}
almost surely.
In that case, there is a trichotomy which dictates whether $\lim_{t\to\infty}\xi(t) =\infty$ almost surely, $\lim_{t\to\infty}\xi(t)   = -\infty$ almost surely or $\limsup_{t\to\infty}\xi(t) = -\liminf_{t\to\infty}\xi(t) = \infty$ accordingly as $\chi'(0)>0$, $<0$ or $=0$, respectively.

\end{proposition}

The leading eigenvalue also features in the following probabilistic
result, which identifies a martingale (also known as the generalised Wald martingale) and associated exponential change of measure
corresponding to an Esscher-type transformation
of a L\'evy process; cf.\ \cite[Proposition XI.2.4, Theorem XIII.8.1]{Asm-apq2}.

\begin{proposition}
\label{p:mg and com}
Let $\mathcal{G}_{t} = \sigma\{(\xi(s), J(s)): s\leq t\}$, $t\geq 0$, and
\begin{equation} M(t,\gamma) =  {\rm e}^{\gamma (\xi(t)-\xi(0)) - \chi(\gamma)t}
    \frac{v_{J(t)}(\gamma)}{v_{J(0)}(\gamma)} ,
  \for t \ge 0, 
  \label{MAPCOM}\end{equation}
for some $\gamma$ such that $\chi(\gamma)$  is defined.
Then,
% \begin{enumerate}[(i)]
% \item
  $M(\cdot,\gamma)$ is a unit-mean martingale with respect to $\GGt$. Moreover, under the change of measure 
  \[
  \left.\frac{{\rm d}\mathbf{P}^\gamma_{x,i}}{{\rm d}\mathbf{P}_{x,i}}\right|_{\mathcal{G}_t} = M(t,\gamma),\qquad t\geq 0,
  \]
  the process $(\xi,J)$ remains in the class of MAPs and, where defined, its  characteristic exponent is given by 
  \begin{equation}
  \boldsymbol{F}_\gamma(z) = \boldsymbol{\Delta}_{\boldsymbol v}(\gamma)^{-1}\boldsymbol{F}(z+\gamma)\boldsymbol{\Delta}_{\boldsymbol v}(\gamma) - \chi(\gamma)\mathbf{I}, 
  \label{Esscher}
  \end{equation}
  where $\mathbf{I}$ is the identity matrix and $\boldsymbol{\Delta}_{\boldsymbol v}(\gamma) = {\rm diag}(\boldsymbol{v}(\gamma))$. It is straightforward to deduce that, when it exists,  the associated leading eigenvalue associated to $  \boldsymbol{F}_\gamma(z)$ is given by
  $\chi_\gamma(z) = \chi(z+ \gamma) - \chi(\gamma)$.
  \end{proposition}

% Making use of this, the following proposition can be obtained; the properties
% of $\chi$ given here are often used in the literature, but for convenience,
% we also provide a short 

 The following properties of $\chi$, lifted from \cite[Proposition 3.4]{KKPW}, will also prove useful in relating the last two results together.

\begin{proposition}
Suppose that $\boldsymbol F$ is defined in some open interval $D$ of $\RR$.
Then, the leading eigenvalue $\chi$ of $\boldsymbol F$
is smooth and convex on $D$.
\end{proposition}

On account of the fact that $\boldsymbol{F}(0) = \boldsymbol{Q}$, it is easy to see that we always have $\chi(0)=0$. If we  assume that the equation
\begin{equation}\label{eq:perron}
\chi(z)=0 
\end{equation}
has a non-zero real root,  henceforth denoted by $\theta$ and referred to as the {\it Cram\'er number}, then the previous proposition allows us to conclude that $\chi$ is defined at least on the interval between $0$ and $\theta$.

If $\theta>0$, then $\chi'(0+)$ is well defined and convexity dictates that it must be negative. In that case $\lim_{t\to\infty}\xi(t) =-\infty$ almost surely. Moreover, if we take $\gamma = \theta$ in Proposition \ref{p:mg and com}, then, as $\chi_\theta'(0-) = \chi'(\theta-)>0$, under the associated change of measure,  $\lim_{t\to\infty}\xi(t) =\infty$ almost surely.

Conversely, if $\theta<0$, then $\chi'(0-)$ is well defined and convexity dictates that it must be positive. In that case $\lim_{t\to\infty}\xi(t) =\infty$ almost surely. Again, if we take $\gamma = \theta$ in Proposition \ref{p:mg and com}, then $\chi_\theta'(0+) = \chi'(\theta+)<0$, under the associated change of measure,  $\lim_{t\to\infty}\xi(t) =-\infty$ almost surely.
In both cases, the change of measure (\ref{MAPCOM}) using $\gamma = \theta$ exchanges the long-term drift of the underlying MAP from $\pm\infty$ to $\mp\infty$.

\subsection{Real self-similar Markov processes}\label{section-rssMp}
In \cite{CPR} the authors confine their attention to rssMp 
in `class \textbf{C.4}'. A rssMp $X$ is in \textbf{C.4} if, for
all $x \ne 0$, $\mathbb{P}_x( \exists t > 0: X_t X_{t -} < 0 ) = 1$;
that is, with probability one, the process $X$ changes
sign infinitely often.
Define
\[ \tau^{\{0\}} = \inf\{t \ge 0: X_t = 0 \},\]
the time to absorption at the origin.

Such a process may be identified with a MAP via a deformation of space and
time  which we call the
{\it Lamperti--Kiu representation} of $X$. The following
result is a simple consequence of \cite[Theorem 6]{CPR}.

\begin{proposition}
  Let $X$ be an rssMp in class \textbf{C.4} and fix $x \ne 0$.
   Then there exists a time-change $\sigma$, adapted to the filtration of $X$,
  such that, under the law $\mathbb{P}_x$, the process 
  \[ (\xi(t),J(t)) = (\log\abs{X_{\sigma(t)}}, {\rm sign}(X_{\sigma(t)})) , \qquad t \ge 0, \]
  is a MAP with state space $E = \{-1,1\}$
  under the law $\mathbf{P}_{\log |x|,{\rm sign}(x)}$.
  Furthermore, the process $X$ under $\mathbb{P}_x$ has the representation
  \[ X_t =  J( \varphi(t)){\rm e}^{ \xi( \varphi(t))
    } , \for 0 \le t < \tau^{\{0\}}, \]
  where $\varphi$ is the inverse of the time-change $\sigma$,
  and may be given by
\begin{equation}\label{e:Lamp time change}
  \varphi(t) = \inf \biggl\{ s > 0 : \int_0^s \exp(\alpha \xi(u))
  \, \dd u > t  \biggr\}, \for t < \tau^{\{0\}}.
\end{equation}
%such that $(\xi, J)$ has law $\mathbf{P}_{\log x, {\rm sign}(x)}$.
\end{proposition}

In short, up to an endogenous time change, a rssMp has a polar decomposition in which $\exp\{\xi\}$ describes the radial distance from the origin and $J$ describes its orientation (positive or negative).

To make the connection with the previous subsection, let us understand how the existence of a Cram\'er number for the underlying MAP to a rssMp affects  path behaviour of the latter.  Revisiting the discussion at the end of the previous subsection, we see that if $\theta >0$ then $\lim_{t\to\infty}\xi(t) =-\infty$. In that case, we deduce from the strong law of large numbers for $\xi$ and the Lamperti--Kiu transform, that
\[
\tau^{\{0\}} = \int_0^\infty {\rm e}^{\alpha\xi(t)}{\rm d}t<\infty\qquad\text{ and }\qquad X_{\tau^{\{0\}}-} = 0
\]
almost surely (irrespective of the point of issue of $X$).
Said another way, the rssMp will be continuously absorbed in the origin after an almost surely finite time. In the case that there is a Cram\'er number which satisfies $\theta<0$, then, again referring to the limiting behaviour of $\xi$ and the Lamperti--Kiu transform, we have 
\[
\tau^{\{0\}} = \int_0^\infty {\rm e}^{\alpha\xi(t)}{\rm d}t=\infty
\]
almost surely (irrespective of the point of issue of $X$). Hence, the associated rssMp never touches the origin. 

{\color{black} We can also reinterpret Proposition \ref{p:mg and com} in light of the Lamperti--Kiu representation and the fact that the quantity $\varphi(t)$ in (\ref{e:Lamp time change}) is also a stopping time. We have that when $\theta>0,$ respectively $\theta<0$,
\begin{equation}
M(\theta, \varphi(t)) = \frac{v_{J(\varphi(t))}(\theta)}{v_{{\rm sign}(x)}(\theta)}{\rm e}^{\theta(\xi(\varphi(t)) - \log |x|)} \mathbf{1}_{(\varphi(t)<\infty)}= \frac{v_{{\rm sign}(X_t)}(\theta)}{v_{{\rm sign}(x)}(\theta)}\frac{|X_t|^\theta}{|x|^\theta}
\mathbf{1}_{(t<\tau^{\{0\}})}, \qquad t\geq 0.\label{itisamartingale}
\end{equation}
is a $\mathbb{P}_x$-martingale, respectively, a  $\mathbb{P}_x$-supermartingale.}

%%%%%%%%%%%%%%%%%%%%%%%%%%%%
%%%%%%%%%%%%%%%%%%%%%%%%%%%%
%%%%%%%%%%%%%%%%%%%%%%%%%%%%
%%%%%%%%%%%%%%%%%%%%%%%%%%%%
%%%%%%%%%%%%%%%%%%%%%%%%%%%%
%%%%%%%%%%%%%%%%%%%%%%%%%%%%
%%%%%%%%%%%%%%%%%%%%%%%%%%%%
%%%%%%%%%%%%%%%%%%%%%%%%%%%%

\section{Main results}

Throughout the remainder of the paper we make following assumption.

\bigskip
\begin{itemize}
\item[{\bf (A):}] The process $X$ is a rssMp whose underlying MAP does not have lattice support and has a leading eigenvalue $\chi$ with Cram\'er number $\theta\neq 0$ such that $\chi'(\theta)$ exists in $\mathbb{R}$.
\end{itemize}

\bigskip

\noindent Under this assumption, our objective is to construct conditioned versions of $X$. When $\theta>0$, through a limiting procedure,  we will build the process $X$ conditioned to avoid the origin. Similarly when $\theta<0$, we will build the process $X$ conditioned to continuously absorb at the origin. Accordingly, in both cases, we shall show the existence of a harmonic function for the process $X$ which is used to make a Doob $h$-transform in the representation of the conditioned processes. 

\begin{theorem}\label{one}
Suppose that $X$ is a rssMp under assumption (A) and $\mathcal{F}_t: = \sigma(X_s: s\leq t)$, $t\geq 0$ is its natural filtration. Define
\[
h_\theta(x): = v_{{\rm sign}(x)}(\theta)|x|^\theta, \qquad x\in\mathbb{R},
\]
and, for Borel set $D$, let $  \tau^{D}:=\inf\{s\geq 0:X_{s}\in D\}$.
\begin{enumerate}
\item[(a)] If $\theta>0$, then 
\begin{equation}
\mathbb{P}^{\circ}_x(A) %: =\lim_{a\to \infty}\mathbb{P}(A\,|\,\tau^{(-a,a)^{\texttt c}}<\tau^{\{0\}}) 
:= \mathbb{E}_x\left[\frac{h_\theta(X_t)}{h_\theta(x)}\mathbf{1}_{(A, \, t<\tau^{\{0\}})}\right],
\label{exptoh}
\end{equation}
for  $t>0$, $x\neq 0$ and $A\in\mathcal{F}_t$,
defines a probability measure on the canonical space of $X$ such that $(X, \mathbb{P}^\circ_x)$, $x\in\mathbb{R}\backslash\{0\}$, is a rssMp. Moreover, for all $A\in\mathcal{F}_{t}$, 
\begin{equation}
\mathbb{P}^{\circ}_x(A) =\lim_{a\to \infty}\mathbb{P}_x( A{\color{black}\cap\{t<\tau^{(-a,a)^{\rm c}}\}}\,|\,\tau^{(-a,a)^{\texttt c}}<\tau^{\{0\}}). 
\label{limitavoid}
\end{equation}

\item[(b)] If $\theta<0$, then, 
\[
\mathbb{P}^{\circ}_x(A, \, t<\tau^{\{0\}}):= \mathbb{E}_x\left[\frac{h_\theta(X_t)}{h_\theta(x)}\mathbf{1}_{(A, \, t<\tau^{\{0\}})}\right],
\]
 for all $t>0$, $x\neq 0$ and $A\in\mathcal{F}_t$, defines a probability measure on the canonical space of $X$  with cemetery state at $0$ such that $(X, \mathbb{P}^\circ_x)$, $x\in\mathbb{R}\backslash\{0\}$, is a rssMp. Moreover,  for all $t>0$ and $A\in\mathcal{F}_{t}$
\begin{equation}
 \mathbb{P}^{\circ}_x(A,\, t< \tau^{\{0\}})=\lim_{a\to 0}\mathbb{P}_x(A\,|\,\tau^{(-a,a)}<\infty). 
\label{limitabsorb}
\end{equation}
\end{enumerate}

\end{theorem}

In case (a) of the above theorem, as $\theta >0$, the Doob $h$-transform rewards paths that drift far from the origin. Indeed the limiting procedure (\ref{limitavoid}) conditions the paths of the rssMp to explore further and further distances from the origin before being absorbed at the origin. In this sense, we refer to the process described in part (a) as the rssMp {\it conditioned to avoid the origin}. In case (b) of the theorem above, the Doob $h$-transform rewards paths that stay close to the origin. Moreover, the limiting procedure (\ref{limitabsorb}) conditions  the paths of the rssMp to ultimately visit smaller and smaller balls centred around the origin. We therefore refer to the process described in part (b) as the rssMp {\it conditioned to absorb continuously at the origin}.

The above theorem constructs the conditioned processes via limiting spatial requirements. For the case of conditioning to avoid the origin, we can give a second sense in which the Doob $h$-transform emerges as the result of a conditioning procedure. The latter is done by conditioning the first visit to the origin to occur later and later in time. 

\begin{theorem}\label{two}
Suppose that $X$ is a rssMp under assumption (A) and $\theta>0$. Then for  $x\in\mathbb{R}\backslash\{0\}$ $t>0$, and $A\in\mathcal{F}_{t}$, we have 
\begin{equation}
\mathbb{P}^{\circ}_x(A ) =\lim_{s\to \infty}\mathbb{P}(A \,|  \tau^{\{0\} } >t+s),
\label{limitavoid2}
\end{equation}
where $\mathbb{P}^\circ_x$, $x\in\mathbb{R}\backslash\{0\}$, is given by \eqref{exptoh}.
\end{theorem}

\section{Remarks on the case of stable processes}

The central family of examples which fits the settting of the two main theorems above is that of a (strictly) stable process with index $\alpha\in(0,2)$. Recall that the latter processes are those rssMp which do not have continuous paths and which are also in the class of L\'evy processes. As a L\'evy process,  a stable process has characteristic exponent $\Psi(\theta): = -t^{-1}\log\mathbb{E}_0[{\rm e}^{{\rm i}\theta X_t}]$, $\theta\in\mathbb{R}$, $t>0$, given by 
\[
\Psi(\theta) = |\theta|^\alpha ({\rm e}^{\pi\iu\alpha(\frac{1}{2} -\rho)} \mathbf{1}_{(\theta>0)} + {\rm e}^{-\pi\iu\alpha (\frac{1}{2} - \rho)}\mathbf{1}_{(\theta<0)}), \qquad \theta\in\mathbb{R},
\]
where $\rho : = \mathbb{P}_0(X_1>0)$. For convenience, we assume throughout this section that $\alpha\rho\in(0,1)$, which is to say that the stable process has path with discontinuities of both signs.

 For such processes, the matrix exponent of the underlying MAP in the Lamperti--Kiu representation has been computed in \cite{KKPW}, with the help of computations in \cite{CPR}, and takes the form 
\begin{equation}
  \boldsymbol{F}(z) =\left[
  \begin{array}{cc}
    - \dfrac{\Gamma(\alpha-z)\Gamma(1+z)}
      {\Gamma(\alpha\hat\rho-z)\Gamma(1-\alpha\hat\rho+ z)}
    & \dfrac{\Gamma(\alpha-z)\Gamma(1+z)}
      {\Gamma(\alpha\hat\rho)\Gamma(1-\alpha\hat\rho)}
    \\
    &\\
    \dfrac{\Gamma(\alpha-z)\Gamma(1+ z)}
      {\Gamma(\alpha\rho)\Gamma(1-\alpha\rho)}
    & - \dfrac{\Gamma(\alpha-z)\Gamma(1+z)}
      {\Gamma(\alpha\rho-z)\Gamma(1-\alpha\rho+z)}
  \end{array} 
  \right],
  \label{MAPHG}
\end{equation}
for $\Re(z)\in(-1,\alpha)$,  where $\hat\rho = 1-\rho$.

A straightforward computation shows that, for $\Re(z)\in (-1,\alpha)$,
\[
{\rm det}\boldsymbol{F}(z)  =\frac{\Gamma(\alpha-z)^2\Gamma(1+z)^2}{\pi^2}\left\{
  \sin(\pi(\alpha\rho- z))\sin(\pi(\alpha\hat\rho- z)) 
  - \sin(\pi \alpha \rho) \sin(\pi\alpha\hat\rho)\right\},
\]
which has a root at $z = \alpha-1$. In turn, this implies that {\color{black}$\chi(\alpha-1)=0$}. One also easily checks with the help of the reflection formula for gamma functions that 
\[
\boldsymbol{v}({\alpha -1}) \propto\left[
\begin{array}{c}
\sin(\pi\alpha\hat\rho)\\
\sin(\pi\alpha\rho)
\end{array} \right].
\]
%and, by recalling that  $\boldsymbol{F}(0) = \boldsymbol{Q}$,
%\begin{equation}
%\boldsymbol{\pi} \propto%\frac{1}{\sin(\pi\alpha\rho) + \sin(\pi\alpha\hat\rho)} 
%\left[
%\begin{array}{c}
%\sin(\pi\alpha\rho)\\
%\sin(\pi\alpha\hat\rho) 
%\end{array}
%\right].
%\label{pi}
%\end{equation}
In that case, we see that Theorems \ref{one} and \ref{two} justify the claim that {\color{black}the family of measures $(\mathbb{P}^\circ_x, x\in\mathbb{R})$ defined via the relation}
\[
\left.\frac{{\rm d}\mathbb{P}^\circ_x}{{\rm d}\mathbb{P}_x}\right|_{\mathcal{F}_t}: = \frac{\sin(\pi\alpha\hat\rho)\mathbf{1}_{(X_t>0)} + \sin(\pi\alpha\rho)\mathbf{1}_{(X_t<0)}}{\sin(\pi\alpha\hat\rho)\mathbf{1}_{(x>0)} + \sin(\pi\alpha\rho)\mathbf{1}_{x <0)}}\left|\frac{X_t}{x}\right|^{\alpha-1}\mathbf{1}_{(t<\tau^{\{0\}})},\qquad t\geq 0,
\]
is the the Doob $h$-transform corresponding to the stable process conditioned to avoid the origin when $\alpha\in(1,2)$, and the stable process conditioned to be continuously absorbed at the origin when $\alpha\in(0,1)$. The former of these two conditionings has already been observed in \cite{CPR}, the latter is a new observation. Note that, when $\alpha=1$, the Doob $h$-transform corresponds to no change of measure at all, as the density is equal to unity and $\tau^{\{0\}} = \infty$ almost surely under $\mathbb{P}_x$, $x\in\mathbb{R}$.

\bigskip

One may prove Theorems \ref{one} and \ref{two} for stable processes 
%in a  slightly stronger sense. Indeed, for $A\in\mathcal{F}_t$, $t\geq 0$, when $\alpha\in(0,1)$
%\[
% \mathbb{P}^{\circ}_x(A, \, t<\tau^{\{0\}})=\lim_{a\to 0}\mathbb{P}_x(A\,|\,\tau^{(-a,a)}<\infty)
%\]
%and, when $\alpha\in(1,2)$, 
%\[
%\mathbb{P}^{\circ}_x(A) =\lim_{a\to \infty}\mathbb{P}_x(A\,|\,\tau^{(-a,a)^{\texttt c}}<\tau^{\{0\}}).
%\]
%This is because one may 
by appealing to a direct form of reasoning using Bayes formula,  scaling, dominated convergence using the fact that $\mathbb{E}_x[|X_t|^{\alpha-\varepsilon}]<\infty$, $x\in\mathbb{R}, t>0, 0<\varepsilon<\alpha$, and the representation of the probabilities:
\[
\mathbb{P}_x(\tau^{(-1,1)^{{\rm c}}}<\tau^{\{0\}}) = (\alpha-1) x^{\alpha-1}\int_1^{1/x} (t-1)^{\alpha\rho-1}(t+1)^{\alpha\rhohat-1}{\rm d}t, \qquad x\in (0,1)
\]
for $\alpha\in(1,2)$, and 
\[
\mathbb{P}_x(\tau^{(-1,1)}<\infty) = \frac{\Gamma(1-\alpha\rho)}{\Gamma(\alpha\rhohat)\Gamma(1-\alpha)}
  \int^1_{\frac{x-1}{x+1}} t^{\alpha\rhohat - 1} (1-t)^{-\alpha} \, \dd t ,\qquad x>1
\]
for $\alpha\in(0,1)$.
The first of these probabilities is taken from Corollary 1 of \cite{PS} and the second from Corollary 1.2 of \cite{KPW}. %{\color{black} Is it true that once can use DCT without restricting $X$ in the limiting procedure of the conditioning? Please can you verify Pite without including computations here (but you should include them in your PhD writeup).}

For the general case, no such detailed formulae are available and a different approach is needed.  The main point of interest is in understanding the asymptotic probabilities  of the conditioning event in Theorems \ref{one} and \ref{two} by appealing to a Cram\'er-type result for the decay of the probabilities $\mathbb{P}_x(\tau^{(-a,a)}<\infty)$ and $\mathbb{P}_x(\tau^{(-a,a)^{{\rm c}}}<\infty)$ as $a\to\infty.$

\bigskip

An additional point of interest  in the case of stable processes pertains to the setting of the   so-called Riesz--Bogdan--Zak transform; see \cite{BZ} and \cite{K}. The understanding of $\mathbb{P}^\circ_x$, $x\in\mathbb{R}\backslash\{0\}$, gives context to the transformation. 

\begin{theorem}[Riesz--Bogdan--Zak transform]
\label{th:BZ} Suppose that $X$ is a stable process with $\alpha\in(0,2)$ satisfying $\alpha\rho\in(0,1)$.
Define
\[
\eta(t) = \inf\{s>0 : \int_0^s |X_u|^{-2\alpha}{\rm d}u >t\}, \qquad t\geq 0.
\]
Then, for all $x\in\mathbb{R}\backslash\{0\}$, $(-1/{X}_{\eta(t)})_{t\geq 0}$ under $\mathbb{P}_{x}$ is equal in law to $(X, \mathbb{P}_{-1/x}^\circ)$. Moreover, the process $(X, \mathbb{P}^\circ_x)$, $x\in\mathbb{R}\backslash\{0\}$ is a self-similar Markov process with underlying MAP via the Lamperti-Kiu whose Matrix exponent satisfies, for $\Re(z)\in(-\alpha,1)$,
\begin{equation}
\boldsymbol{F}^\circ(z) =
 \left[
  \begin{array}{cc}
    - \dfrac{\Gamma(1-z)\Gamma(\alpha+z)}
      {\Gamma(1-\alpha\rho-z)\Gamma(\alpha\rho+ z)}
    & \dfrac{\Gamma(1-z)\Gamma(\alpha+z)}
      {\Gamma(\alpha\rho)\Gamma(1-\alpha\rho)}
    \\
    &\\
    \dfrac{\Gamma(1-z)\Gamma(\alpha+ z)}
      {\Gamma(\alpha\hat\rho)\Gamma(1-\alpha\hat\rho)}
    & - \dfrac{\Gamma(1-z)\Gamma(\alpha+z)}
      {\Gamma(1-\alpha\hat\rho-z)\Gamma(\alpha\hat\rho+z)}
  \end{array} 
  \right].
  \label{Fcirc}
\end{equation}

\end{theorem}

\section{Cram\'er-type results for MAPs and the proof of Theorem \ref{one}}\label{Cramer}

Appealing to the Lamperti-Kiu process, we note that, for $|x|<a$ 
\[
\mathbb{P}_x(\tau^{(-a,a)^{{\rm c}}}<\tau^{\{0\}}) = \mathbf{P}_{\log|x|, {\rm sign}(x)}(T^+_{\log a}<\infty) = 
\mathbf{P}_{0, {\rm sign}(x)}(T^+_{\log (a/|x|)}<\infty)
\]
where $T^+_{y} = \inf\{t>0 : \xi(t)>y\}$. A similar result may be written for $\mathbb{P}_x(\tau^{(-a,a)}<\infty)$, albeit using $T^-_{y} := \inf\{t>0 : \xi(t)<y\}$. %, except in terms of the first passage time 
% Similarly, we have, for $|x|>a$,
%\[
%\mathbb{P}_x(\tau^{(-a,a)^{{\rm c}}}<\infty) = \mathbf{P}_{(\log|x|, {\rm sign}(x))}(T^-_{\log a}<\infty) = 
%\mathbf{P}_{\log (|x|/a), {\rm sign}(x)}(T^-_{0}<\infty),
%\]
%where 
% $T^-_y = \inf\{t>0: \xi_t<y\}$ with $y = \log(a/|x|)$. 
This suggests that the asymptotic behaviour of the two probabilities of interest can be studied through the behaviour of the underlying MAP. In fact, it turns out that, in both cases, a Cram\'er-type result in the MAP context provides the desired asymptotics.

\begin{proposition}\label{regvar}Suppose that $X$ is a rssMp under assumption (A).
\begin{itemize}
\item[(a)] When  $\theta>0$, there exists a constant $C_\theta\in(0,\infty)$ such that, for $|y|>0$ 
\[
\lim_{a\to\infty}{a}^{\theta}\mathbb{P}_y(\tau^{(-a,a)^{{\rm c}}}<\tau^{\{0\}}) =v_{{\rm sign}(y)}(\theta)C_\theta |y|^\theta.
\]
In particular,  
\begin{equation}
\lim_{a\to\infty}\frac{\mathbb{P}_y(\tau^{(-a,a)^{{\rm c}}}<\tau^{\{0\}})}{\mathbb{P}_x(\tau^{(-a,a)^{{\rm c}}}<\tau^{\{0\}})} = \lim_{a\to\infty}\frac{ \mathbf{P}_{0, {\rm sign}(y)}(T^+_{\log (a/|y|)}<\infty)}{ \mathbf{P}_{0, {\rm sign}(x)}(T^+_{\log (a/|x|)}<\infty)} = \frac{v_{{\rm sign}(y)}(\theta)}{v_{{\rm sign}(x)}(\theta)}\left|\frac{y}{x}\right|^\theta, \qquad x,y\in\mathbb{R}.
\label{SV1}
\end{equation}
%For each $b>0$, the convergence is uniform for $|y|<b$. % {\color{black}Do we really need to mention the uniformity now that it does not get used in the proof?}
%{\color{black} Moreover, for each $x,y\in\mathbb{R}$ sufficiently large
%\[
%\frac{v_{{\rm sign}(x)}(\theta)}{v_{{\rm sign}(y)}(\theta)}\left|\frac{x}{y}\right|^\theta\frac{\mathbb{P}_y(\tau^{(-a,a)^{{\rm c}}}<\tau^{\{0\}})}{\mathbb{P}_x(\tau^{(-a,a)^{{\rm c}}}<\tau^{\{0\}})} 
%\]
%is uniformly bounded from above.}
\item[(b)] When $\theta<0$ , there exists a constant $\tilde{C}_\theta\in(0,\infty)$ such that, for $|y|>0$ 
\[
\lim_{a\to0}{a}^{\theta}\mathbb{P}_y(\tau^{(-a,a)}<\infty) =v_{{\rm sign}(y)}(\theta)\tilde{C}_\theta |y|^\theta.
\]
In particular,  \begin{equation}
\lim_{a\to0}\frac{\mathbb{P}_y(\tau^{(-a,a)}<\infty)}{\mathbb{P}_x(\tau^{(-a,a)}<\infty)} = \lim_{a\to\infty}\frac{ \mathbf{P}_{0, {\rm sign}(x)}(T^-_{\log (a/|y|)}<\infty)}{ \mathbf{P}_{0, {\rm sign}(x)}(T^-_{\log (a/|x|)}<\infty)} = \frac{v_{{\rm sign}(y)}(\theta)}{v_{{\rm sign}(x)}(\theta)}\left|\frac{y}{x}\right|^\theta, \qquad x,y\in\mathbb{R}.
\label{SV2-andreas}
\end{equation}
%For each $b>0$, the convergence is uniform for each $|y|>b$. %{\color{black}Do we really need to mention the uniformity now that it does not get used in the proof?}
\end{itemize}
\end{proposition}

This result will be proved below after some preliminary lemmas. % However,  let us first give some indication of what is needed in order consider the case of $\mathbf{P}_{0, {\rm sign}(x)}(T^+_{y}<\infty)$. 
Recalling the discussion from \cite{K}, an excursion theory for MAPs reflected in their running maxima exists with strong similarities to the case of L\'evy processes. Specifically, there is a MAP, say $(H^+(t), J^+(t))_{t\geq 0}$, with the property that $H^+ $ is non-decreasing  with the same range as the running maximum process $\sup_{s\leq t}\xi(s)$, $t\geq 0$. Moreover, the trajectory of the associated Markov chain $J^+$ agrees with the chain $J$ on the times of increase of the  running maximum. We also refer to the Appendix in \cite{DDK} for further information on classical excursion theory for MAPs.

As an increasing MAP, the process $(H^+, J^+)$ has associated to it a number of characteristics. When $J^+=\pm1$, the process $H^+$ has the increments of a subordinator with drift $\delta_{\pm1}$ and L\'evy measure $\Upsilon_{\pm1}$ and is sent to a cemetery state $\{+\infty\}$ at rate $q_{\pm1}$. When $J^+$ jumps from $i$ to $j$ with $i,j\in\{-1,1\}$ and $i\neq j$, the process $H^+$ experiences an independent jump with distribution $F^+_{i,j}$ at rate $\Lambda_{i,j}$. 
For convenience, we will introduce the Laplace matrix exponent $\boldsymbol\kappa$ in the form 
\[
\mathbb{E}_{0,i}[{\rm e}^{-\lambda H^+(t)}; J^+(t) = j] = [{\rm e}^{-\boldsymbol{\kappa}(\lambda) t }]_{i,j}, \qquad \lambda \geq 0,
\]
where, 
\[
\boldsymbol{\kappa}(\lambda) = {\rm diag}(\Phi_1(\lambda), \Phi_{-1}(\lambda)) - \boldsymbol{\Lambda}\circ{\boldsymbol K}(\lambda), \qquad \lambda \geq 0,
\]
where  for $i =\pm1$, $\Phi_i(\lambda)$ is the Laplace exponent of the subordinator encoding the dynamics of $H$ when $J^+  = i$, $\boldsymbol\Lambda$ is the intensity matrix of $J^+$ and ${\boldsymbol K}(\lambda)_{i,j} =\int_{(0,\infty)} {\rm e}^{-\lambda x} F^+_{i,j}({\rm d}x)$ for $i,j=\pm1$ with $i\neq j$ and otherwise ${\boldsymbol K}(\lambda)_{i,i} =1$, for $i=\pm1$.

In the next lemma we write  the crossing probability of interest in terms of the potential measures 
\[
 U_{i,j}^+({\rm d}x)=\int_{0}^{\infty}\mathbf{P}_{0,i}(H^+(t)\in {\rm d}x,J^+(t)=j){\rm d}s, \qquad x\geq 0, i,j\in\{-1,1\}.
 \]

\begin{lemma}\label{creeplemma} The probability of first passage over a threshold  can be decomposed into the probability of creeping and the probability of jumping over it. %For $i=\pm1$, write $\delta_i\geq 0$ for the drift coefficient of $\Phi_i$.
\item[(a)] For $y>0$,
\begin{equation}\label{eq:last}
\mathbf{P}_{0,i}(T_{y}^+ <\infty,H^+_{T_y^+}>y)=\sum_{ j,k =\pm1}\int_{0}^{y}U^+_{i,j}({\rm d}z)\left[\mathbf{1}_{(k\neq j)}\Lambda_{j,k}\overline{F}^+_{j,k}(y-z)+\mathbf{1}_{(k=j)}\overline{\Upsilon}_{j}(y-z)\right].
\end{equation}
\item[(b)] If $\delta_j>0$  for some $j = \pm1,$ % we have $\delta_i>0$ or $\Pi_i(0,\infty)=\infty$.
then $U_{i,j}^+$ has a density on  $[0,\infty)${\color{black}\ for $i=\pm1$,  which has a continuous version, say $u^+_{i,j}$}. Moreover, for $y>0$,
$$p_i(y): = \mathbf{P}_i(T_y^+<\infty,H^+(T_y^+)=y)=\sum_{j=\pm1} \delta_j u^+_{i,j}(y),\qquad y> 0, i=\pm1,$$
where we understand $u^+_{i,j}\equiv 0$ if $\delta_j =0$. If $\delta_j=0$ for both $j = \pm1$, then $p_i(y) = 0$ for all $y>0$.
\end{lemma}

\begin{proof}
(a) Appealing to the {\color{black} compensation formula for the} Cox process that describes the  jumps in $H^+$, we may write for $y>0$,
\begin{eqnarray}
\mathbf{P}_{0,i}(T_{y}^+ <\infty)&=&\mathbf{E}_{0,i}\left[\sum_{0<s<\infty} \mathbf{1}_{(y-H^+(s-)>0)}\mathbf{1}_{(y-H^+(s)<0)}\right]\notag\\&=&
\sum_{j,k=\pm1}\mathbf{E}_{0,i}\left[\sum_{0<s<\infty} \mathbf{1}_{(y-H^+(s-)>0)}\mathbf{1}_{(\Delta H^+(s)>y-H^+(s-))}\mathbf{1}_{(J^+({s-})=j,J^+(s)=k)}\right]\notag\\
&=&\sum_{j,k=\pm1}\mathbf{1}_{(j\neq k)}\int_{0}^{\infty}\mathbf{P}_{0,i}(H^+(s-)<y,J^+({s-})=j)\Lambda_{j,k}\overline{F}^+_{j,k}(x-H^+(s-)){\rm d}s\notag\\
&&+\sum_{j=\pm1}\int_{0}^{\infty}\mathbf{P}_{0,i}(H^+(s-)<y,J^+({s-})=j) 
\overline{\Upsilon}_{j}(x-H^+(s-)){\rm d}s,
\label{firstpassage}
\end{eqnarray}
where $\Delta H^+(s)  = H^+(s) - H^+(s-)$,  $\overline{F}^+_{j,k}(x) = 1-F^+_{j,k}(x)$ and $\overline{\Upsilon}_{j}(x) =\Upsilon_j(x,\infty)$.
%Introducing the potential measure 
%\[
% U_{i,j}^+({\rm d}x)=\int_{0}^{\infty}\mathbf{P}_{0,i}(H^+(t)\in {\rm d}x,J^+(t)=j){\rm d}s, \qquad x\geq 0.
% \]
{\color{black}When we express the right-hand side of (\ref{firstpassage})  in terms of the potential measure we  get (\ref{eq:last}).}

(b) We first define, for $a>0$,
\begin{equation} M_i(a):=\int_0^a \mathbf{P}_{0,i}(H(T_y^+)=y,T_y^+<\infty)dy = \int_0^a p_i(y)dy.
\label{absctsM}
\end{equation}
The analogue of the L\'evy-It\^o decomposition for subordinators tells us that, up to killing at rate $q_{\pm1}$, when $J^+$ is in state $\pm1$,
$$ H^+(t)=\int_0^{t}\delta_{J(t)}dt+\sum_{0<s<t}\Delta H^+(s), \qquad t\geq 0.$$
Then,
$$ M_i(a)=\mathbf{E}_{0,i}\left[H^+(T_{a}^+ -) -\sum_{0<s<T_{a}^+}(H^+(s)-H^+(s-)) \,;\, T^+_a <\infty\right]=\mathbf{E}_{0,i}\left[\int_0^{T_a^+}\delta_{J^+(t)}dt\,;\, T^+_a <\infty \right].$$
Hence, for $a>0$,
$$ M_i(a)=\mathbf{E}_i\left[\int_0^{\infty}\mathbf{1}_{(0\leq H^+(t) \leq a)}\delta_{J(t)}\dd t\right]=\sum_{j=\pm1} \delta_j U^+_{i,j}[0,a].$$
Noting from \eqref{absctsM} that  that $M_i$ is almost everywhere differentiable on $(0,\infty)$, the above equality  tells us that, for $j$ such that $\delta_j\neq 0$ the potential measure $U^+_{i,j}$ has a density. %So, the right-hand side is differentiable.
Otherwise, if $\delta_j=0$ for both $j  = \pm 1$, then $p_i(y) = 0$ for Lebesgue almost every $y>0$.

 We define, for each $i,j = \pm 1$ and $x>0$,
$$ p_{i,j}(x)=\mathbf{P}_{0,i}(T_{x}^+ <\infty,H^+(T_x^+)=x,J^+(T_x^+)=j)  \text{ such that  } p_i(x)=\sum_{j=\pm1} p_{i,j}(x).$$
Fix $i\in\{-1,1\}$. We want to show that $p_i(x)$ is continuous. For that, we shall use the fact that  
\begin{equation}\label{eq:smallCreep}
\lim_{\epsilon\downarrow 0} p_{i,j}(\epsilon)=\mathbf{1}(\delta_i >0)\mathbf{1}(i=j).
\end{equation}
This is due to the fact that the stopping time $T:=\inf\{s>0:J^+(s)\neq i \text{ or }H^+(s)=+\infty\}$ is exponentially distributed while the time $T^+_{\epsilon}\downarrow 0$ as $\epsilon\downarrow 0$. Hence, on $\{t<T\}$, $H^+_t$ behaves as a (killed) L\'evy subordinator and so $T^+_{\epsilon}< T$ with increasing probability, tending to 1 as $\epsilon\downarrow0$.  Hence, the result follows from the classical case of L\'evy subordinators; see \cite{kesten}.

{\color{black}By the Markov property we} have the lower bound
\begin{equation}\label{eq: below}
p_{i}(x+\epsilon)\geq \mathbf{P}_{0,i}(H^+(T_{x}^+)=x,H^+(T_{x+\epsilon}^+)=x+\epsilon,T_{x+\epsilon}^+<\infty)=\sum_{j= \pm1}p_{i,j}(x)p_j(\epsilon) .
\end{equation}
If we take the limit $\epsilon\downarrow 0$ and use \eqref{eq:smallCreep}, then we have that
$$\lim_{\epsilon\downarrow 0}p_i(x+\epsilon)\geq \sum_{j=\pm1} p_{i,j}(x)\mathbf{1}(\delta_j>0)=\sum_{j=\pm1}p_{i,j}(x)=p_i(x).$$
%In the last equality, we have used the fact that $p_{i,j}(x)\leq p_i(x)=0$ if $\delta_\pm =0$. 
On the other hand, we can split the behavior of creeping over $x+\epsilon$ into two types 
\begin{eqnarray*}
p_i(x+\epsilon)&= &\mathbf{P}_{0,i}(H^+(T_{x}^+)=x,H^+(T_{x+\epsilon}^+)=x+\epsilon,T_{x+\epsilon}^+<\infty)\\
&& +\mathbf{P}_{0,i}(H^+(T_x^+)>x,H^+(T_{x+\epsilon}^+)=x+\epsilon,T_{x+\epsilon}^+<\infty).
\end{eqnarray*} 
The first probability on the right-hand side above corresponds to the right-hand side of (\ref{eq: below}) and we can bound the second term by the event that $\{0<O_x\leq\epsilon\}$, where we define the overshoot $O_x:=H^+(T_x^+)- x%H^+(T_x^+ -)
$. Hence, we deduce that
\begin{eqnarray*}
p_i(x+\epsilon)\leq\sum_{j=\pm1}p_{i,j}(\epsilon)p_j(x) +\mathbf{P}_{0,i}(O_x\in (0,\epsilon]).
\end{eqnarray*} 
The second probability on the right-hand side above goes to zero as $\epsilon\to0$. If we now combine this inequality with \eqref{eq: below} and take the limit $\epsilon\downarrow 0$, then we can then show that $$\lim_{\epsilon\downarrow 0}p_i(x+\epsilon)=p_i(x)=\sum_{j=\pm1}p_{i,j}(x).$$
We can also  show in a similar way  that  $\lim_{\epsilon\downarrow 0}p_i(x-\epsilon)=p_i(x)$ and hence $p_i$ is  continuous. Note that the preceding reasoning is  valid without discrimination for the case that $p_i$ is almost everywhere equal to zero. The proof is now complete. 
\end{proof}

Understanding the asymptotic of $\mathbf{P}_{0,i}(T_{y}^+ <\infty)$ is now a matter of Markov additive renewal theory. In this respect, let us say some more words about the the Markov additive renewal measure $U_{i,j}$.

We will restrict the forthcoming discussion to the setting that $\theta >0$. Recall from the discussion at the end of Section \ref{ss:MAP} that this implies $\lim_{t\to\infty}\xi(t) = -\infty$, where $\xi$ is the  MAP underlying the rssMp. A consequence of this observation is that the process $H^+$ experiences killing. To be more precise it has   killing rates which we previously denoted by  $q_{\pm1}>0$. This makes the measures $U_{i,j}^+$ finite. As with classical renewal theory, we can use the existence of the Cram\'er number $\theta$ to renormalise the measures $U^+_{i,j}$ so that they are appropriate for use with asymptotic Markov additive renewal theory. 

Appealing to the exponential change of measure described in Proposition \ref{p:mg and com}, we note that the law of $(H^+,J^+)$ under $\mathbf{P}_{0,i}^\theta$ satisfies 
\[
\mathbf{P}_{0,i}^{\theta}(H^+(t) \in {\rm d}x, J^+(t)= j) = \frac{v_j(\theta)}{v_i(\theta)}{\rm e}^{\theta x}\mathbf{P}_{0,i}(H^+(t) \in {\rm d}x, J^+(t)= j), \qquad i,j=\pm1, x\geq 0.
\]
In particular, the role of $\boldsymbol\kappa$ for $(H^+, J^+)$ under $\mathbb{P}^\theta_{0,i}$, $i=\pm1$ is played by 
\[
\boldsymbol{\kappa}_\theta(\lambda) = \boldsymbol{\kappa}(\lambda -\theta ), \qquad \lambda\geq 0.
\]
Hence, we have that
\[
U^{\theta, +}_{i,j}({\rm d}x) : = \int_0^\infty \mathbf{P}_{0,i}^{\theta}(H^+(t) \in {\rm d}x, J^+(t)= j) {\rm d}t = \frac{v_j(\theta)}{v_i(\theta)}{\rm e}^{\theta x}U^+_{i,j}({\rm d}x), \qquad x\geq 0.
\]
Again,  referring to the discussion at the end of  Section \ref{ss:MAP}, since $\lim_{t\to\infty}\xi(t) = \infty$ almost surely under {\color{black}$\mathbf{P}^{\theta}_{0,i}$}, we may now claim that the adjusted Markov additive renewal measure $U^\theta_{i,j}({\rm d}x)$ is that of an unkilled subordinator MAP.

\begin{lemma}\label{cramerestimate} Suppose that $\theta>0$.
There exists a constant $C_\theta>0$, such that, as $y\to\infty$, 
$$ {\rm e}^{\theta y}\mathbf{P}_{0,i}(T_{y}^+ <\infty)\to v_i(\theta) C_\theta. $$

%{\color{black} In particular,  $\sup_{y>0}{\rm e}^{\theta y}\mathbf{P}_{0,i}(T_{y}^+ <\infty)$ is uniformly bounded from above and away from zero.}
\end{lemma}
\begin{proof}
Picking up equation (\ref{eq:last}), we  have, for $i = \pm1$, 
\begin{eqnarray}
\lefteqn{
{\rm e}^{\theta y}\mathbf{P}_{0,i}(T_{y}^+ <\infty,\, H^+_{T_y^+}>y)}
\notag\\
&&=v_i(\theta)\sum_{ j,k =\pm1}\int_{0}^{y}{\rm e}^{\theta(y-z)}\frac{1}{v_j(\theta)}U^{\theta,+}_{i,j}({\rm d}z)\left[\mathbf{1}_{(k\neq j)}\Lambda_{j,k}\overline{F}^+_{j,k}(y-z)+\mathbf{1}_{(k=j)}\overline{\Upsilon}_{j}(y-z)\right].
\label{rewrite}
\end{eqnarray}
\noindent Our aim is to convert this into a form that we can apply the discrete-time Markov Additive Renewal Thoerem  \ref{prop} in  the Appendix.

To this end, we  define the sequence of random times $\Theta_1,\Theta_2,\cdots$ such that $\Theta_{i+1}-\Theta_i$ are independent and  exponentially distributed with parameter $1$. For convenience, define $\Theta_0 = 0$. We want to relate $(H^+,J^+)$ to a discrete-time Markov additive renewal process $(\Xi_n, M_n)$,  $n\geq 0$, such that % $\Xi_n: = \sum_{k = 0}^n \Delta_k$ and 
\[
 \Delta_n:=\Xi_{n+1} - \Xi_n = H^+_{\Theta_{n+1}}-H^+_{\Theta_n} \text{ and } M_n=J^+({\Theta_n}), \qquad n\geq 0 .\]
 A future quantity of interest is the stationary mean increment $\mu^+_\theta: = \mathbf{E}^\theta_{0,\boldsymbol{\pi}^{\theta}}[H_1(\Theta_1)]$, where $\boldsymbol{\pi}^{\theta} = (\pi^{\theta}_{1}, \pi^{\theta}_{-1})$ is the stationary distribution of $J$ (and hence of $J^+$ since it is described pathwise by J sampled at a sequence of stopping times) under $\mathbf{P}^\theta$. In this respect, we note from Corollary 2.5 in Chapter XI of \cite{Asm-apq2} that, % if  $\eta_i: =\mathbf{E}^\theta_{0,i}[H^+({\Theta_1})]$, then the  
\begin{eqnarray}
\mu^+_\theta  %&=&\mathbf{E}^\theta_{0,i}[H^+({\Theta_1})]\\
&=&\int_0^\infty {\rm e}^{-t}\mathbf{E}^\theta_{0,\boldsymbol{\pi}^{\theta}}[H^+(t)]{\rm d}t \notag\\
&=&\int_0^\infty {\rm e}^{-t}[\chi^+_\theta(0)t + \boldsymbol{\pi}^{\theta}\cdot \boldsymbol{k}^{\theta} -  \boldsymbol{\pi}^{\theta}\cdot {\rm e}^{\boldsymbol{\Lambda}^\theta t}\boldsymbol{k}^{\theta}]{\rm d}t \notag\\
&=& \chi^+_\theta(0) +  \boldsymbol{\pi}^{\theta}\cdot \boldsymbol{k}^{\theta} -  \boldsymbol{\pi}^{\theta}\cdot (\boldsymbol{\Lambda}^\theta - \boldsymbol{I})^{-1}\boldsymbol{k}^{\theta},\label{similarlater}
%&=&\sum_{j=1}^N\int_0^\infty {\rm e}^{-t} [\boldsymbol{\kappa}_\theta'(0+){\rm e}^{\boldsymbol{\Lambda}^\theta t}]_{i,j}{\rm d}t\\
%&=&\sum_{j=1}^N \sum_{\ell = 1}^N[\boldsymbol{\kappa}'(-\theta)]_{i,\ell} [\boldsymbol{\Lambda}^\theta - I]_{\ell, j}\\
%&=&[\boldsymbol{\kappa}'(-\theta)(\boldsymbol{\Lambda}^\theta - I)]{\mathbf{1}},
\end{eqnarray}
where $\chi_\theta^+(0)$ is the leading eigenvalue of $\boldsymbol{\kappa}_\theta(0)$, $\boldsymbol{k}^\theta=\boldsymbol{v}'(\theta)$ and $\boldsymbol{\Lambda}^\theta = \boldsymbol{\kappa}_\theta(0)$. All of these quantities are guaranteed to exist thanks to the assumption (A); see for example Section 2 of Chapter XI in \cite{Asm-apq2}.

Note, moreover, that 
\begin{eqnarray}
U^{\theta,+}_{i,j}({\rm d}x)& = & \int_{0}^{\infty}\mathbf{P}^\theta_{0,i}(H_{t}^+\in {\rm d}x,\, J^+(t)=j){\rm d}t\notag\\
& = & \sum_{n =1}^{\infty}\int_{0}^{\infty}{\rm e}^{-t} \frac{t^{n-1}}{(n-1)!} \mathbf{P}^\theta_{0,i}(H_{t}^+\in {\rm d}x,\, J^+(t)=j){\rm d}t\notag\\
%& = & \sum_{n=1}^{\infty}\mathbb{P}^{0,i}(\Theta_n\in {\rm d}t) \mathbb{P}^{0,i}(H_{t}^+\in {\rm d}x)\notag\\
&=& \sum_{n=1}^{\infty}\mathbf{P}^\theta_{0,i}(H_{\Theta_n}\in {\rm d}x,J_{\Theta_n}=j)\notag\\
&=:&R_{i,j}({\rm d}x)-\delta_0({\rm d}x)\mathbf{1}(i=j),\label{similarspirit}
\end{eqnarray} 
where, on the right-hand side, we have used the notation of the discrete-time Markov additive renewal measure in the Appendix.
 
 Turning back to (\ref{rewrite}), if we define 
\begin{equation}
g_j(x)=\sum_{k=\pm1}\frac{1}{v_j(\theta)}{\rm e}^{\theta x}\left[\mathbf{1}(k\neq j)\Lambda_{j,k}\overline{F}_{j,k}(x)+\mathbf{1}(k=j)\overline{\Upsilon}_{j}(x)\right],\qquad x\geq 0,
\label{checkint}
\end{equation}
for $j =\pm1$,
 then, as soon as we can verify that these functions are   %{\color{black}directy(?)%} 
directly  Riemann integrable, then we can apply the conclusion of Theorem \ref{prop} in the Appendix and conclude that 
\begin{eqnarray*}
\lefteqn{
\lim_{y\to\infty}{\rm e}^{\theta y}\mathbf{P}_{0,i}(T_{y}^+ <\infty,\, H^+_{T_y^+}>y)}
\\
%&&=\sum_{ j,k =\pm1}\lim_{y\to\infty}\int_{0}^{y}{\rm e}^{\theta(y-z)}\frac{v_i(\theta)}{v_j(\theta)}U^{\theta,+}_{i,j}(dz)\left[\mathbf{1}_{(k\neq j)}\Lambda_{j,k}\overline{F}^+_{j,k}(y-z)+\mathbf{1}_{(k=j)}\overline{\Upsilon}_{j}(y-z)\right]\\
&&=v_i(\theta)\sum_{ j,k =\pm1}\frac{\pi^{\theta}_j}{v_j(\theta)\mu^+_\theta}\int_{0}^{\infty}{\rm e}^{\theta s}\left[\mathbf{1}_{(k\neq j)}\Lambda_{j,k}\overline{F}^+_{j,k}(s)+\mathbf{1}_{(k=j)}\overline{\Upsilon}_{j}(s)\right]{\rm d}s,
\end{eqnarray*}
where $\pi^{\theta}_{j}$, $j=\pm1$ is the stationary distribution of the chain $J^+$ under {\color{black}$\mathbf{P}^{\theta}_{x,i}$}, $x\in\mathbb{R}$, $i=\pm1$.
Note, moreover that, from Lemma \ref{creeplemma}, together with Theorem 1.2 of  \cite{Als2014},
\[
{\rm e}^{\theta y}\mathbb{P}_i(T_y^+<\infty,H^+_{T_y^+}=y)=v_i(\theta)\sum_{j=\pm1}\frac{1}{v_j(\theta)}\delta_j u^{\theta, +}_{i,j}(y)\to v_i(\theta)\sum_{j = \pm1} \delta_j \frac{\pi^{\theta}_j}{v_j(\theta)\mu^+_\theta},
\]
as $y\to\infty$.

To finish the proof we must thus  verify the direct Riemann integrability of $g_j(x)$, $j = \pm1$ in (\ref{checkint}). Note however, that $g_j(x)$ is the product of ${\rm e}^{\theta x}$ and a monotone decreasing function, hence it suffices to check that $\int_0^\infty g_j(x){\rm d}x<\infty$, $j = \pm1$. To this end, remark that, for $\lambda$ in the domain where $\boldsymbol{\kappa}$ is defined, 
\[
(\boldsymbol{\kappa}(\lambda)\boldsymbol{1})_j = q_j + \delta_j\lambda +  \int_0^\infty (1- {\rm e}^{-\lambda x})\Upsilon_j({\rm d}x) + {\color{black}\sum_{k=\pm1}}\mathbf{1}_{(j\neq k)}\Lambda_{j,k}\int_0^\infty {\rm e}^{-\lambda x}F_{j,k}(\dd x), \qquad j = \pm1.
\]
In particular, with an integration by parts, we have 
\[
\frac{q_j - (\boldsymbol{\kappa}(-\theta)\boldsymbol{1})_j }{\theta} =\delta_j +  \int_{0}^{\infty}{\rm e}^{\theta s}\left[{\color{black}\sum_{k=\pm1}}\mathbf{1}_{(k\neq j)}\Lambda_{j,k}\overline{F}^+_{j,k}(s)+\mathbf{1}_{(k=j)}\overline{\Upsilon}_{j}(s)\right]{\rm d}s,\qquad j = \pm1,
\]
where the left-hand side is finite thanks to the assumption (A). This completes the proof, albeit to note that 
\[
\lim_{y\to\infty}{\rm e}^{\theta y}\mathbb{P}_i(T_y^+<\infty) = v_i(\theta)\sum_{j = \pm1}\frac{\pi^{\theta}_j[q_j - (\boldsymbol{\kappa}(-\theta)\boldsymbol{1})_j ]}{\theta v_j(\theta)\mu^+_\theta},
\]
which identifies explicitly the constant $C_\theta$ in the statement of the lemma.
\end{proof}

\begin{proof}[Proof of Proposition \ref{regvar}] 
First assume that $\theta>0$.
A particular consequence of Lemma \ref{cramerestimate} is that 
\[
\lim_{a\to\infty}\frac{\mathbb{P}_y(\tau^{(-a,a)^{{\rm c}}}<\tau^{\{0\}})}{\mathbb{P}_x(\tau^{(-a,a)^{{\rm c}}}<\tau^{\{0\}})} = \lim_{a\to\infty}\frac{ \mathbf{P}_{0, {\rm sign}(y)}(T^+_{\log (a/|y|)}<\infty)}{ \mathbf{P}_{0, {\rm sign}(x)}(T^+_{\log (a/|x|)}<\infty)} = \frac{v_{{\rm sign}(y)}(\theta)}{v_{{\rm sign}(x)}(\theta)}\left|\frac{y}{x}\right|^\theta, \qquad x,y\in\mathbb{R}.
\]
%The limiting behaviour shows that the ratios of probabilities on the  left-hand side can be seen as a regularly varying function in $|y|$. On account of the fact that $\theta>0$, Theorem 1.5.2. of \cite{BGT} tells us that this regular variation must be uniform on all sets of the form $|y|<b$ for $b>0$.

\bigskip

Now we turn our attention to the case that $\theta<0$. We appeal to duality and write 
\[
\mathbb{P}_x(\tau^{(-a,a)}<\infty) = \mathbf{P}_{(\log|x|, {\rm sign}(x))}(T^-_{\log a}<\infty) =  \tilde{\mathbf{P}}_{(-\log|x|, {\rm sign}(x))}(T^+_{-\log a}<\infty), 
\]
where under $\tilde{\mathbf{P}}_{x,i}$, $x\in\mathbb{R}$, $i=\pm1$, is the law of $(-\xi, J)$. Note, the associated matrix exponent of this process is $\tilde{\boldsymbol F}(z):={\boldsymbol F}(-z)$, whenever the right-hand side is well defined. In particular, we note that $\tilde{\boldsymbol F}(-\theta) =0$, which is to say that $-\theta>0$ is the Cram\'er number for the process $(-\xi, J)$. Moreover, 
$
\tilde{\boldsymbol F}(-\theta){\boldsymbol v}(\theta):={\boldsymbol F}(\theta){\boldsymbol v}(\theta) = 0, 
$
which is to say that  $\tilde{\boldsymbol v}(-\theta) = {\boldsymbol v}(\theta)$.
The first part of the proof can now be re-cycled to deduce the conclusions in part (b) of the statement of the proposition.
\end{proof}

\begin{proof}[Proof of Theorem \ref{one}] % It suffices to give the proof of part (a), the proof of part (b) follows by analogy. 
%We start by noting that, 
%\[
%\frac{h_\theta(X_t)}{h_\theta(x)}\mathbf{1}_{(t<\tau^{\{0\}})}, \qquad t\geq 0
%\]
%is a unit-mean positive martingale thanks to \eqref{MAPCOM} and the Lamperti--Kiu representation of $X$. 
The (super)martingale (\ref{itisamartingale}) applies an exponential change of measure to $(\xi, J)$, albeit on the sequence of stopping times $\varphi(t)$, for $t< \tau^{\{0\}}$. As the change of measure (\ref{MAPCOM}) keeps $(\xi,J)$ in the class of MAPs, it follows that $\mathbb{P}^\circ_x$, $x\in\mathbb{R}\backslash\{0\}$, corresponds to the law of a rssMp. 

In the case of (a), recalling the discussion preceding Section \ref{section-rssMp}, the underlying MAP for $(X, \mathbb{P}^\circ_x)$, $x\in\mathbb{R}\backslash\{0\}$ drifts to $+\infty$. This means that under the change of measure, $X$ is a rssMp that never touches the origin, i.e. it is a conservative process. In the case of (b), the underlying MAP drifts to $-\infty$ and hence, under the change of measure $X$ is continuously absorbed at the origin, so it is non-conservative.

%{\color{black} Do we need this paragraph?} Now note that, as a function of $y$, the left-hand side of   (\ref{SV1}) is a regularly varying function in $|y|$ with index $\theta>0$ as $|y|$ tends to $0$. By Theorem 1.5.2 of \cite{BGT}, this implies that it is uniformly regularly varying in $|y|$ on all intervals of the form $(0,b]$ for any $b>0$ as $|y|$ tends to 0. 
{\color{black} For the proof of (a), we  follow a standard line of reasoning that can be found, for example, in \cite{Chaumont96}. Appealing to the Markov property, self-similarity,  Fatou's Lemma and  (\ref{SV1}),  we have, for $A\in \mathcal{F}_t$,
\begin{eqnarray*}
\lefteqn{
{\color{black}\liminf_{a\to \infty}} \mathbb{P}_x(A{\color{black}\cap\{t<\tau^{(-a,a)^{\rm c}}\}}\,|\,\tau^{(-a,a)^{\rm c}}<\tau^{\{0\}})
}\notag\\
&& {\color{black}=\liminf_{a\to \infty}}\mathbb{E}_{x}\left[\mathbf{1}_{(A, \, t<\tau^{\{0\}}\wedge \tau^{(-a,a)^{\rm c}})}\frac{\mathbb{P}_{X_{t}}(\tau^{(-a,a)^{\rm c}}<\tau^{\{0\}})}{\mathbb{P}_{x}(\tau^{(-a,a)^{\rm c}}<\tau^{\{0\}})}\right]\notag\\
&&\geq \mathbb{E}_{x}\left[\mathbf{1}_{(A, \, t<\tau^{\{0\}})} {\color{black}\liminf_{a\to \infty}}\frac{\mathbb{P}_{a^{-1}X_{t}}(\tau^{(-1,1)^{\rm c}}<\tau^{\{0\}})}{\mathbb{P}_{a^{-1}x}(\tau^{(-1,1)^{\rm c}}<\tau^{\{0\}})}\right]\notag\\
&&=\mathbb{E}_x\left[\mathbf{1}_{(A, \, t< \tau^{\{0\}})}\frac{h_\theta(X_t)}{h_\theta(x)}\right].
\label{doagain}
\end{eqnarray*}
Recalling the martingale property from (\ref{itisamartingale}) together with the above inequality, but now applied to the event $A^c$, tells us that 
\begin{eqnarray*}
\lefteqn{
{\color{black}\limsup_{a\to \infty}}\mathbb{P}_x(A{\color{black}\cap\{t<\tau^{(-a,a)^{\rm c}}\}}\,|\,\tau^{(-a,a)^{\rm c}}<\tau^{\{0\}})
}\\
&&{\color{black}\leq } 1- {\color{black}\liminf_{a\to \infty}}\mathbb{P}_x(A^c{\color{black}\cap\{t<\tau^{(-a,a)^{\rm c}}\}}\,|\,\tau^{(-a,a)^{\rm c}}<\tau^{\{0\}})\\
&&{\color{black}\leq} \mathbb{E}_x\left[\frac{h_\theta(X_t)}{h_\theta(x)}\mathbf{1}_{( t< \tau^{\{0\}})}\right] - \mathbb{E}_x\left[\frac{h_\theta(X_t)}{h_\theta(x)}\mathbf{1}_{(A^c, \, t< \tau^{\{0\}})}\right]\\
&&=\mathbb{E}_x\left[\frac{h_\theta(X_t)}{h_\theta(x)}\mathbf{1}_{(A, \, t< \tau^{\{0\}})}\right]
\end{eqnarray*}
and the required limiting identity follows.
%
%The previously described uniform regular variation, the fact that we consider the set $A$ on the event $\{t<\tau^{(-b,b)^c}\}$ and dominated convergence allow us to deduce that 
%\[
%\lim_{a\to \infty}\mathbb{P}_x(A,\, t< \tau^{(-b,b)^{\rm c}}\,|\,\tau^{(-a,a)^{\rm c}}<\tau^{\{0\}})= 
%\mathbb{E}_x\left[\frac{h_\theta(X_t)}{h_\theta(x)}\mathbf{1}_{(A, \, t< \tau^{\{0\}})}\right]
%\]
%as required. 
} 

{\color{black}The proof of (b) is similar to that of (a) except that in this case (\ref{itisamartingale}) ensures that $X^{\theta}_{t}$ is a super-martingale only and hence the final part of the argument above does not extend to this setting. To overcome this difficulty we proceed as follows. Notice $\tau^{(-a,a)}\to\tau^{\{0\}}$ as $a\to 0.$ As before for $A\in \mathcal{F}_t$, we have
\begin{eqnarray*}
\lefteqn{
{\color{black}\liminf_{a\to 0}}\mathbb{P}_x(A{\color{black}\cap\{t<\tau^{(-a,a)}\}}\,|\,\tau^{(-a,a)}<\infty)
}\notag\\
&& {\color{black}=\liminf_{a\to 0}}\mathbb{E}_{x}\left[\mathbf{1}_{(A, \, t<\tau^{(-a,a)})}\frac{\mathbb{P}_{X_{t}}(\tau^{(-a,a)}<\infty)}{\mathbb{P}_{x}(\tau^{(-a,a)}<\infty)}\right]\notag\\
&&\geq \mathbb{E}_{x}\left[\mathbf{1}_{(A, \, t<\tau^{\{0\}})} {\color{black}\liminf_{a\to 0}}\frac{\mathbb{P}_{a^{-1}X_{t}}(\tau^{(-1,1)}<\infty)}{\mathbb{P}_{a^{-1}x}(\tau^{(-1,1)}<\infty)}\right]\notag\\
&&=\mathbb{E}_x\left[\mathbf{1}_{(A, \, t< \tau^{\{0\}})}\frac{h_\theta(X_t)}{h_\theta(x)}\right].
\label{doagainagain}
\end{eqnarray*}
Now, the second half of the argument in (a) extends to this setting if the following equation holds true 
\begin{equation*}
\lim_{a\to 0}\mathbb{P}_x(t<\tau^{(-a,a)}\,|\,\tau^{(-a,a)}<\infty)=\mathbb{E}_x\left[\frac{h_\theta(X_t)}{h_\theta(x)}\mathbf{1}_{(t< \tau^{\{0\}})}\right].
\end{equation*}
On the one hand, the Markov property Fatou's lemma and the estimate (\ref{SV2-andreas}) imply that
\begin{equation*}
\begin{split}
\liminf_{a\to 0}\mathbb{P}_x(t<\tau^{(-a,a)}\,|\,\tau^{(-a,a)}<\infty)&=
\liminf_{a\to 0}\mathbb{P}_x\left(\mathbf{1}_{(t<\tau^{(-a,a)})}\frac{\mathbb{P}_{X_{t}}(\tau^{(-a,a)}<\infty)}{\mathbb{P}_{x}(\tau^{(-a,a)}<\infty)}\right)\\ 
&\geq \mathbb{E}_x\left[\frac{h_\theta(X_t)}{h_\theta(x)}\mathbf{1}_{(t< \tau^{\{0\}})}\right].
\end{split}
\end{equation*}
Now, the estimate in (b) in Proposition (\ref{regvar}) 
implies that for $y\neq 0$ 
$$\lim_{a\to 0}\left(\frac{a}{|y|}\right)^{\theta}\mathbb{P}_{y}(\tau^{(-a,a)}<\infty)=\lim_{a\to 0}\left(\frac{a}{|y|}\right)^{\theta}\mathbb{P}_{\sgn(y)}(\tau^{(-\frac{a}{|y|},\frac{a}{|y|})}<\infty)=v_{{\rm sign}(y)}(\theta)\tilde{C}_\theta,$$ and the convergence holds uniformly in $a/|y|$ such that ${a}/{|y|}<\epsilon,$ for $\epsilon>0.$ Moreover, for ${a}/{|y|}>\epsilon$ the term
$({a}/{|y|})^{\theta}\mathbb{P}_{y}(\tau^{(-a,a)}<\infty)$
remains bounded. Thus for $x\neq 0, \epsilon>0,$ fixed we have
\begin{equation*}
\begin{split}
&\limsup_{a\to 0}\mathbb{P}_x\left(\mathbf{1}_{(t<\tau^{(-a,a)})}\frac{\mathbb{P}_{X_{t}}(\tau^{(-a,a)}<\infty)}{\mathbb{P}_{x}(\tau^{(-a,a)}<\infty)}\right)\\
&=\limsup_{a\to 0}\mathbb{P}_x\left(\mathbf{1}_{((a/|X_{t}|)<\epsilon,\ t<\tau^{(-a,a)})}\frac{\mathbb{P}_{X_{t}}(\tau^{(-a,a)}<\infty)}{\mathbb{P}_{x}(\tau^{(-a,a)}<\infty)}\right)\\ 
&\quad+\limsup_{a\to 0}\mathbb{P}_x\left(\mathbf{1}_{((a/|X_{t}|)\geq \epsilon,\ t<\tau^{(-a,a)})}\frac{\mathbb{P}_{X_{t}}(\tau^{(-a,a)}<\infty)}{\mathbb{P}_{x}(\tau^{(-a,a)}<\infty)}\right)\\
&= \mathbb{E}_x\left[\frac{h_\theta(X_t)}{h_\theta(x)}\mathbf{1}_{(t< \tau^{\{0\}})}\right]+\limsup_{a\to 0}\mathbb{P}_x\left(\mathbf{1}_{((a/|X_{t}|)\geq \epsilon,\ t<\tau^{(-a,a)})}\frac{\mathbb{P}_{X_{t}}(\tau^{(-a,a)}<\infty)}{\mathbb{P}_{x}(\tau^{(-a,a)}<\infty)}\right).
\end{split}
\end{equation*}
Finally the limsup in the above estimate is equal to zero because it can be bounded by above as follows
\begin{equation*}
\begin{split}
&\mathbb{P}_x\left(\mathbf{1}_{((a/|X_{t}|)\geq \epsilon,\ t<\tau^{(-a,a)})}\frac{\mathbb{P}_{X_{t}}(\tau^{(-a,a)}<\infty)}{\mathbb{P}_{x}(\tau^{(-a,a)}<\infty)}\right)\\
&\leq \frac{1}{a^{\theta}\mathbb{P}_{x}(\tau^{(-a,a)}<\infty)}\mathbb{P}_x\left(\mathbf{1}_{((a/|X_{t}|)\geq \epsilon,\ t<\tau^{(-a,a)})}|X_{t}|^{\theta}\sup_{|z|\geq \epsilon}|z|^{\theta}\mathbb{P}_{\sgn(z)}(\tau^{(-z,z)}<\infty)\right)\\
&=\frac{x^{\theta}\sup_{|z|\geq \epsilon}|z|^{\theta}\mathbb{P}_{\sgn(z)}(\tau^{(-z,z)}<\infty)}{a^{\theta}\mathbb{P}_{x}(\tau^{(-a,a)}<\infty)}\mathbb{P}^{\circ}_x\left(\mathbf{1}_{(a/|X_{t}|\geq \epsilon,\ t<\tau^{(-a,a)})}\frac{v_{{\rm sign}(x)}(\theta)}{v_{{\rm sign}(X_{t})}(\theta)}\right),
\end{split}
\end{equation*}
and by the monotone convergence theorem the rightmost term in the above inequality tends to zero when $a\to 0.$
}
\end{proof}

\section{Integrated exponential MAPs and the proof of Theorem \ref{two}}

In order to approach the asymptotic conditioning in Theorem \ref{two}, we need to understand the tail behaviour of the probabilities $\mathbb{P}_x(\tau^{\{0\}}>t)$, as $t\to\infty$, for all $x\neq0$. Indeed, following arguments in the spirit of the proof of Theorem \ref{one}, for $A\in \mathcal{F}_t$, the Markov property tells us that, for $x\in\mathbb{R}\backslash\{0\}$,
\begin{equation}
\lim_{s\to \infty}\mathbb{P}_x(A \,|  \tau^{\{0\} } >t+s)=\lim_{s\to\infty}\mathbb{E}_x\left[\mathbf{1}(A, \, t<\tau^{\{0\}})\frac{\mathbb{P}_{X_t}(\tau^{\{0\}}>s)}{\mathbb{P}_x(\tau^{\{0\}}>t+s)}\right].
\label{conditionedtime}
\end{equation}
We are thus compelled to consider  the asymptotic behaviour of 
$ \mathbb{P}_x(\tau^{\{0\}}>t)$ as  $t\to\infty.$
In particular, we will prove the following {\color{black} result, which is also of intrinsic interest.}

\begin{theorem}\label{tailtheorem} Define 
\[
I = \int_0^\infty {\rm e}^{\alpha\xi(s)}\dd s.
\]
For $\theta>0$, we have that $\mathbf{E}_{0,k}[I^{\theta/\alpha-1}]<\infty$, $k=\pm1$, and 
\[
 \mathbb{P}_x(\tau^{\{0\}}>t) \sim v_{{\rm sign}(x)}(\theta) |x|^{\theta}t^{-\theta/\alpha}\sum_{j=\pm1}\pi^\theta_{j}\frac{\mathbf{E}_{0,j}[I^{\theta/\alpha-1}]}{\mu_{\theta}|\alpha-\theta|v_j(\theta)}, \text{ as }t\to\infty,
\]
where $\mu_{\theta}=\sum_{j = \pm1}\pi^\theta_j\mathbf{E}_{0,j}^{\theta}[\xi(1)]$ and $\boldsymbol{\pi}^\theta=(\pi^\theta_{1}, \pi^\theta_{-1})$ is the stationary distribution of $J$ under $\mathbf{P}^\theta_{x,i}$, $x\in\mathbb{R}$, $i=\pm1$. 
%Moreover, we can show that$$ \mathbb{P}_x(\tau^{\{0\}}>t)=\mathbb{P}_i(I>x^{-\alpha}t)\sim v_i(\theta) x^{\theta} l^{-\theta/\alpha}\sum_{k\in E}\frac{\mathbb{E}_k(I^{\theta/\alpha-1})}{\alpha\mu_{\theta}},$$which shows that this type of conditioning is the same as the type we worked out previously.
\end{theorem}

The Lamperti--Kiu representation $(\xi,J)$ allows us to write
\begin{equation}
\tau^{\{0\}}=|x|^{\alpha}\int_0^{\infty}{\rm e}^{\alpha\xi(s)}{\rm d}s:=|x|^{\alpha}I.
\label{T0andI}
\end{equation}
Note in particular that Cram\'er's condition and the assumption $\theta>0$ ensures that $\xi'(0)<0$. In turn, this implies  that $\xi(t)\to-\infty$ almost surely at a linear rate which guarantees the almost sure finiteness of $I$.
It follows that $\mathbb{P}_x(\tau^{\{0\}}>t) = \mathbf{P}_{0, {\rm sign}(x)}(I>|x|^{-\alpha} t)$ and hence, to prove the above theorem, we need to first pass through some technical results concerning the tail of the distribution of $I$.

The asymptotic behaviour of the tail distribution of objects similar to $I$, when the process $\xi$ is replaced by a L\'evy process, has been considered in \cite{R, AR}. We will borrow some of the ideas from the second of these two papers and apply them  in he Markov additive setting in proving \ref{tailtheorem}. 
To this end, let us introduce the potential measure 
$$ V_{i,j}({\rm d}x)=\int_{0}^{\infty}\mathbf{P}_{0,i}(\xi(s) \in {\rm d}x,J(s) =j){\rm d}s.$$

\begin{proposition} For $t>0$ and $i=\pm1$, 
\begin{equation}
\mathbf{P}_{0,i}(I>t){\rm d}t=\sum_{j=\pm1} \int_{\mathbb{R}} V_{i,j}({\rm d}y){\rm e}^{\alpha y}\mathbf{P}_{0,j}({\rm e}^{\alpha y}I \in {\rm d}t).
\label{takeLT}
\end{equation}
%Further, if the probability density of $I$ exists for all $t>0$, in which case it is written 
%$ k_i(t)$,
% then we have
%$$\mathbb{P}_i(I>t)=\sum_{j\in E}\int_{\mathbb{R}}V_{i,j}({\rm d}y)k_{j}(t{\rm e}^{-\alpha y}). $$
\end{proposition}
\begin{proof}\label{eq:Ilap} The method of proof is to show the left- and right-hand sides of (\ref{takeLT}) are equal by considering their Laplace transforms.
Integration by parts shows us that, for $\lambda>0$, we have on the one hand,
\begin{equation}
\mathbf{E}_{0,i}(1-{\rm e}^{-\lambda I})=\lambda \int_0^{\infty}{\rm e}^{-\lambda t}\mathbf{P}_{0,i}(I>t){\rm d}t .\label{comparelater}
\end{equation}
We shall use the above equation for comparison later. On the other hand, we have for $\lambda>0$, 
\begin{eqnarray}
\mathbf{E}_{0,i}(1-{\rm e}^{-\lambda I})&=&\mathbf{E}_{0,i}\left[\int_0^{\infty}{\rm d}({\rm e}^{-\lambda\int_s^{\infty}{\rm e}^{\alpha\xi(u)}{\rm d}u})\right]\notag\\
&=&\lambda\mathbf{E}_{0,i}\left[\int_0^{\infty}{\rm e}^{\alpha \xi(s)}{\rm e}^{-\lambda\int_s^{\infty}{\rm e}^{\alpha\xi(u)}{\rm d}u}{\rm d}s\right]\notag\\
&=& \lambda \int_0^{\infty}\sum_{j=\pm1} \mathbf{E}_{0,i}\left[{\rm e}^{\alpha \xi(s)}{\rm e}^{-\lambda {\rm e}^{\alpha \xi(s)}\int_s^{\infty}{\rm e}^{\alpha(\xi(u)-\xi(s))}{\rm d}u};J(s)=j\right]{\rm d}s\notag\\
 &=& \lambda \sum_{j=\pm1}\int_0^{\infty} \mathbf{E}_{0,i}\left[{\rm e}^{\alpha \xi(s)}\mathbf{E}_{0,j} \left.\left[{\rm e}^{-\lambda {\rm e}^{\alpha y} I}\right]\right|_{y= \xi(s)}\right]{\rm d}s\notag\\
&=&\lambda\sum_{j=\pm1} \int_{\mathbb{R}}V_{i,j}({\rm d}y){\rm e}^{\alpha y} \mathbf{E}_{0,j}[{\rm e}^{-\lambda {\rm e}^{\alpha y} I}]\notag\\
&=&\lambda\sum_{j=\pm1} \int_{\mathbb{R}}V_{i,j}({\rm d}y){\rm e}^{\alpha y} \int_0^{\infty}\mathbf{P}_{0,j}({\rm e}^{\alpha y}I\in {\rm d}t) {\rm e}^{-\lambda t},\label{secondversion}
\end{eqnarray}
where we have applied the conditional stationary independent increments of $(\xi, J)$ in the fourth equality.
Now comparing (\ref{secondversion}) with (\ref{comparelater}), we see that
$$ \mathbf{P}_{0,i}(I>t){\rm d}t=\sum_{j=\pm1} \int_{\mathbb{R}} V_{i,j}({\rm d}y){\rm e}^{\alpha y}\mathbf{P}_{0,j}({\rm e}^{\alpha y}I \in {\rm d}t),$$
for $t>0$, as required.
\end{proof}

Now that we have expressed the tail probabilities $\mathbf{P}_{0,i}(I>t)$ in terms of the potential measure $V_{i,j}$, we may again turn to renewal theory for Markov additive random walks in order to extract the desired asymptotics as $t\to\infty$. With a view to applying Theorem \ref{prop} in the Appendix, let us therefore introduce $(M_n,\Delta_n)$ defined as 
$$ M_n=J({\Theta_n}) \text{ and } \Delta_n=\xi_{\Theta_n}, \qquad n\geq 0,$$ where, as before, $\Theta_0=0$ and $\Theta_n$ is the sum of an independent sequence of exponential random variables with unit mean. As in the Appendix, we write $R_{i,j}({\rm d}x)$ for the renewal measure of $(\Xi, M)$, where $\Xi_0 = 0$, $\Xi_n = \Delta_1+\cdots\Delta_n$, $n\geq 1$. We also introduce 
\[
R^\theta_{i,j}({\rm d}x):=\frac{v_j(\theta)}{v_i(\theta)}{\rm e}^{\theta x}R_{i,j}({\rm d}x), \qquad x\in\mathbb{R}.
\]
We note again that $V_{i,j}({\rm d}x) = R_{i,j}({\rm d}x) - \delta_0({\rm d}x)\mathbf{1}_{(i=j)}$.

In a similar spirit to (\ref{similarspirit}), we may use these Markov additive random walks to write for any interval $A \subseteq[0,\infty)$
\begin{eqnarray}
{\rm e}^{(\theta-\alpha)t}\int_{A{\rm e}^{\alpha t}}\mathbf{P}_{0,i}(I>s){\rm d}s&=&\sum_{j=\pm1} \int_{\mathbb{R}} V_{i,j}({\rm d}y){\rm e}^{\alpha y}{\rm e}^{(\theta-\alpha)t}\int_{A{\rm e}^{\alpha t}}\mathbf{P}_{0,j}({\rm e}^{\alpha y}I \in {\rm d}s)\notag\\
&=&v_i(\theta)\sum_{j=\pm1}  \frac{1}{v_j(\theta)}\int_{\mathbb{R}}R^{\theta}_{i,j}({\rm d}y){\rm e}^{(\theta-\alpha)(t-y)}\int_{A{\rm e}^{\alpha t}}\mathbf{P}_{0,j}({\rm e}^{\alpha y}I \in {\rm d}s)\notag\\
&&+\mathbf{1}_{(i=j)}{\rm e}^{(\theta-\alpha)t}\mathbf{P}_{0,j}(I \in A{\rm e}^{\alpha t})\notag\\
&=&v_i(\theta)\sum_{j=\pm1}  \frac{1}{v_j(\theta)}\int_{\mathbb{R}}R^{\theta}_{i,j}({\rm d}y){\rm e}^{(\theta-\alpha)(t-y)}\mathbf{P}_{0,j}(I \in A{\rm e}^{\alpha(t-y)})\notag\\
&&-\mathbf{1}_{(i=j)}{\rm e}^{(\theta-\alpha)t}\mathbf{P}_{0,j}(I \in A{\rm e}^{\alpha t}).\label{pre-MART}
\end{eqnarray}
Noting that the main term on the right-hand side above is a convolution between the renewal measure $R^\theta_{i,j}$ and  the function
\[
g_j(z, A) : = \frac{1}{v_j(\theta)} {\rm e}^{(\theta-\alpha)z}\mathbf{P}_{0,j}(I \in A{\rm e}^{\alpha t}), \qquad z\in\mathbb{R},
\] we are now  almost ready to apply the discrete-ime Markov Additive Renewal Theorem \ref{prop} in the Appendix. It turns out that we need to choose the interval $A$ judiciously according to whether $\theta$ is bigger or smaller than $\alpha$ in order to respect the directly Riemann integrability condition in the renewal theorem.  We therefore digress with an additional technical lemma before returning to the limit in (\ref{pre-MART}) and the proof of Theorem \ref{tailtheorem}.

\begin{lemma}\label{Ifinite} When  $\theta>0$,  $\mathbf{E}_{0,j}\left(I^{\theta/\alpha-1}\right)<\infty$, for all $j=\pm1$.
\end{lemma}

\begin{proof}  When $\theta =\alpha$ the result is trivial. 
The case that  $\theta/\alpha<1$ turns out to be a direct consequence of Proposition 3.6 from \cite{KKPW}.  To be more precise, careful inspection of the proof there shows that (in our terminology) if $0<\alpha\beta\leq \theta$ then $\mathbf{E}_{0,i}[I^{\beta-1}]<\infty$, in which case one takes $\beta = \theta/\alpha$.

    For the final case that $\theta/\alpha>1$, we can replicate the recurrence relation from Section 1.2 of \cite{AR}.  Appealing to (\ref{takeLT}), we have, for $\beta\in (0,\theta/\alpha)$ and $k=\pm1$, 
    \[
     \mathbf{E}_{0,k}[I^{\beta}]= \beta \int_0^{\infty} s^{\beta -1} \mathbf{P}_{0,k}(I>s){\rm d}s= \beta \int_0^{\infty}  s^{\beta -1} \sum_j \int_{\mathbb{R}} V_{k,j}({\rm d}y){\rm e}^{\alpha y}\mathbf{P}_{0,j}({\rm e}^{\alpha y}I \in {\rm d}s).
     \]
Let us momentarily assume that $\mathbb{E}_k[I^{\beta-1}]<\infty$ for $k= \pm1$. We can use Fubini's theorem and put $s=t {\rm e}^{\alpha y}$, and get 
\begin{eqnarray*}
\mathbf{E}_{0,k}[I^{\beta}]&=& \beta\sum_j \int_{\mathbb{R}}{\rm e}^{\alpha\beta y}V_{k,j}({\rm d}y) \int_0^{\infty} t^{\beta-1}\mathbf{P}_{0,j} (I\in {\rm d}t)\\
&=&\beta\sum_j \mathbb{E}_j[I^{\beta-1}]\int_{\mathbb{R}}{\rm e}^{\alpha\beta y}V_{k,j}({\rm d}y)\\
&=& \beta\sum_j \mathbf{E}_{0,j}[I^{\beta-1}]\int_0^{\infty} {\rm d}s\int_{\mathbb{R}}{\rm e}^{\alpha\beta y}\mathbf{E}_{0,k}[\xi(s) \in {\rm d}y,J(s) =j]\\
&=&\beta\sum_j \mathbf{E}_{0,j}[I^{\beta-1}]\int_0^{\infty}(\exp\{tF(\alpha\beta)\})_{k,j} {\rm d}t\\
&=&  \beta\sum_j \mathbf{E}_{0,j}[I^{\beta-1}](\boldsymbol{F}(\alpha\beta)^{-1})_{k,j} .
\end{eqnarray*}
where the right-hand side uses the fact that $\beta\in (0,\theta/\alpha)$. We deduce that $\mathbf{E}_{0,k}[I^{\beta-1}]<\infty$ for $k= \pm1$ implies that $\mathbf{E}_{0,k}[I^{\beta}]<\infty$ for $k= \pm1$.

If $n$ is the smallest non-negative integer such that $\theta/\alpha-n \in (0,1]$, we can use Proposition 3.6 from \cite{KKPW} again, to deduce that $\mathbf{E}_{0,k}[I^{\theta/\alpha-n}]<\infty$. The  argument in the previous paragraph can now be used inductively to conclude that 
$\mathbf{E}_{0,k}[I^{\theta/\alpha-1}]<\infty$.
\end{proof}

\begin{proof}[Proof of Theorem \ref{tailtheorem}.] We break the proof into three cases. 
We start by assuming that $\theta<\alpha$. In that case, referring to \eqref{pre-MART}, we have, assuming the limit exists on the right-hand side,   
\begin{eqnarray}
\lefteqn{\lim_{t\to\infty}{\rm e}^{(\theta-\alpha)t}\int_{0}^{{\rm e}^{\alpha t}}\mathbf{P}_{0,i}(I>s){\rm d}s}&&\notag\\
&&=\lim_{t\to\infty}v_i(\theta)\sum_{j=\pm1}  \frac{1}{v_j(\theta)}\int_{\mathbb{R}}R^{\theta}_{i,j}({\rm d}y){\rm e}^{(\theta-\alpha)(t-y)}\mathbf{P}_{0,j}(I \in [0,{\rm e}^{\alpha(t-y)}])\notag\\
&&=\lim_{t\to\infty}v_i(\theta)\sum_{j=\pm1}  \frac{1}{v_j(\theta)}\int_{\mathbb{R}}R^{\theta}_{i,j}({\rm d}y)
g_j(t-y)
\label{*}
\end{eqnarray}
where
\[
g_k(y)=\frac{1}{v_k(\theta)}{\rm e}^{(\theta-\alpha) y}\int_0^{{\rm e}^{\alpha y}}\mathbf{P}_{0,k}(I\in {\rm d}s) , \qquad k = \pm1, y\in\mathbb{R}.
\]
Note in particular that 
\[
\int_{\mathbb{R}}g_k(y){\rm d}y=\frac{1}{ (\alpha -\theta)v_k(\theta)}\mathbf{E}_{0,k} [I^{\theta/\alpha-1} ], \qquad k=\pm1,
\]
which is finite by Lemma \ref{Ifinite}. Moreover, since $g_k(x)$ is product of an exponential and a monotone function, it is a standard exercise to show that it is also directly Riemann integrable.

The discrete-time Markov Additive Renewal Theorem \ref{prop} in the Appendix now justifies the limit in  \eqref{*} so that 
\begin{eqnarray}
\lim_{t\to\infty}{\rm e}^{(\theta-\alpha)t}\int_{0}^{{\rm e}^{\alpha t}}\mathbf{P}_{0,i}(I>s){\rm d}s&=&v_i(\theta)\sum_{j=\pm1}  \frac{\pi_j^\theta}{\mu_{\theta}|\alpha-\theta|v_j(\theta)}\mathbf{E}_{0,j} [I^{\theta/\alpha-1} ],
\label{logit}
\end{eqnarray}
provided $\mu_\theta<\infty$. This last condition is easily verified as a consequence of assumption (A). Indeed, according to Corollary 2.5 of Chapter XI in  \cite{Asm-apq2}, we have 
\[
\mu_\theta = \chi'(\theta) + \boldsymbol{\pi}^\theta\cdot\boldsymbol{k}^\theta -  \boldsymbol{\pi}^\theta\cdot (\boldsymbol{Q}^\theta- \boldsymbol{I})^{-1}\boldsymbol{k}^\theta,
\]
where $\boldsymbol{Q}^\theta = \boldsymbol{F}_\theta(0)$ is the intensity matrix of $J$ under $\mathbf{P}^\theta$.
 Writing the limit in (\ref{logit}) with a change of variables, we have 
\[
\lim_{u\to\infty}u^{(\theta/\alpha-1 )}\int_{0}^{u}\mathbf{P}_{0,i}(I>s){\rm d}s=v_i(\theta)\sum_{j=\pm1}  \frac{\pi_j^\theta}{\mu_{\theta}|\alpha-\theta|v_j(\theta)}\mathbf{E}_{0,j} [I^{\theta/\alpha-1} ],
\]
{\color{black}which shows, for each $i,$} regular variation of the integral on the left-hand side. Appealing to the monotone density theorem for regularly varying functions, we now conclude that 
\[
\mathbf{P}_{0,i}(I>u)\sim u^{-\theta/\alpha} v_i(\theta)\sum_{j=\pm1}  \frac{\pi_j^\theta}{\mu_{\theta}|\alpha-\theta|v_j(\theta)}\mathbf{E}_{0,j} [I^{\theta/\alpha-1} ], \qquad u\to\infty,
\]
and the result for the case that $\theta<\alpha$ now follows from (\ref{T0andI}).

\bigskip

The proof  for the case $\theta>\alpha$ is completed by starting the reasoning as with the case of $\theta<\alpha$ but with $A = (1,\infty)$ in (\ref{pre-MART}). The desired asymptotics again comes from the first term on the right-hand side of (\ref{pre-MART}) using a similar application of the Markov Additive Renewal Theorem \ref{prop}. The details are left to the reader. The second term on the right-hand side of (\ref{pre-MART}) becomes negligible since 
\[
\lim_{t\to\infty}{\rm e}^{(\theta-\alpha)t}\mathbf{P}_{0,j}(I >{\rm e}^{\alpha t})
 = 0
\]
on account of the fact that $\mathbf{E}_{0,i}[I^{\theta/\alpha -1}]<\infty$.

The case that $\alpha=\theta$ is  dealt with similarly by starting from (\ref{pre-MART}) but now setting $A = (1,\lambda)$ for some $\lambda>1$. In that case, the second term on the right-hand side of (\ref{pre-MART}) makes no contribution to the limit in question since
\[
\lim_{t\to\infty}\mathbf{P}_{0,j}(I > {\rm e}^{\alpha t}) = 0.
\]
The integral in the first term on the right-hand side of (\ref{pre-MART}) can be written in the form 
\[
%\lefteqn{
\int_{\mathbb{R}}R^{\theta}_{i,j}({\rm d}y)\mathbf{P}_{0,j}(I \in A{\rm e}^{\alpha(t-y)})
%}
=\int_{\mathbb{R}}\mathbf{P}_{0,j}(I\in {\rm d}v)R^\theta_{i,j}(t - \alpha^{-1}\log v, t-\alpha^{-1}\log v + \alpha^{-1}\log \lambda).
\]
Thanks to Lemma 3.5 of \cite{Als2014}, we have the uniform estimate 
\[
\sup_{x\in\mathbb{R}}R^\theta_{i,j}(x, x+ \alpha^{-1}\log \lambda)\leq \pi^\theta_jR^\theta_{i,i}(-\alpha^{-1}\log \lambda,  \alpha^{-1}\log \lambda).
\]
This result is accompanied by the classical form of the Markov Additive Renewal Theorem (c.f Theorem 3.1 of \cite{Als2014}), which states that 
\[
\lim_{x\to\infty}R^\theta_{i,j}(x, x+ \alpha^{-1}\log \lambda) = \pi^\theta_j\frac{\log \lambda}{\alpha\mu_\theta}.
\]
This allows us to apply the dominated convergence and note, in conjunction with the classical form of the Markov Additive Renewal Theorem (c.f Theorem 3.1 of \cite{Als2014}) that 
\[
\lim_{t\to\infty}\int_{\mathbb{R}}R^{\theta}_{i,j}({\rm d}y)\mathbf{P}_{0,j}(I \in A{\rm e}^{\alpha(t-y)}) =  \pi^\theta_j\frac{\log \lambda}{\alpha\mu_\theta}.
\]
Plugging this limit back into the first term on the right-hand side of (\ref{pre-MART}) provides the necessary convergence to complete the proof in the same way as the previous two cases. The details are again left to the reader.
\end{proof}

\begin{proof}[Proof of Theorem \ref{two}.] {\color{black}We use similar reasoning to the proof of Theorem \ref{one}
%again appealing to uniform regular variation, 
 to pass the limit through the expectation on the right-hand side of (\ref{conditionedtime}) and derive the result. }
\end{proof}

\section*{Appendix: Markov additive renewal theory}

\renewcommand{\thesubsection}{\Alph{subsection}}

 \renewcommand{\theequation}{A.\arabic{equation}}
  \renewcommand{\thetheorem}{A.\arabic{theorem}}

  % redefine the command that creates the equation no.
  \setcounter{equation}{0}  % reset counter 
    \setcounter{theorem}{0}  % reset counter 

%
%For the sake of convenience, we convert a result for discrete-time Markov additive renewal theory to the setting of continuous-time using stochastic embedding.  In the discrete setting there are a number of existing references, for example \cite{A, Als2014, S}, to name but a few.  The main result of this appendix result below is an adaptation of a classical renewal theorem for discrete-time Markov additive processes. 
%\medskip
%

Consider a discrete-time stochastic process described by the pair $(\Delta, M): =((\Delta_{n},M_{n}) )_{n\geq 0}$, where $\Delta_{n}$ takes real (or just positive) values and $M_n$ takes  values in the set $E : = \{1,2,\dots,N\}$.
We shall specify the law of such a process as follows.

Set $\Delta_0 = 0$. For each $i,j\in E$, there is a probability distribution $P_{i,j}(x)$ such that, 
 conditioning on the history of $(\Delta, M)$ up to time $n-1$, the distribution of $(\Delta_n, M_n)$ is given by
\[
 {\rm P}(M_n=j,\Delta_n \leq x|(M_k,X_k),k=0,\dots,n-1)=P_{M_{n-1},j}(x).
 \]
In this sense, we have that the the process $M = \{M_n: n\geq 0\}$ alone is a Markov chain on $E$ with transition matrix $p_{i,j}: = P_{i,j}(\infty)$, for $i,j\in E$.
The possibility that $p_{ii}>0$ is not excluded here. 

The distribution of $\Delta_n$ only depends on the state at time $n-1$ which makes the discrete-time Markov additive process
\[
\Xi_n: = \sum_{k = 0}^n \Delta_k, \qquad n\geq 0,
\]
the analogue of a Markov additive random walk (or Markov additive renewal process if the increments are all positive). %, for example the process $(H^+,J^+)$ mentioned in Section \ref{Cramer}.

\medskip

To state a classical renewal result for discrete-time Markov additive processes, we need to introduce a little more notation. 
The mean 
 transition is given by 
\[
\eta_i =\sum_{j\in E}\int_{\mathbb{R}}xP_{i,j} ({\rm d}x), \qquad i\in E
\]
Moreover,  the measure $R_{i,j}$ denotes the occupation measure 
\[
R_{i,j}(x)=\sum_{n =1}^\infty {\rm P}(\Xi_n\leq x,\,M_n=j|M_0=i).
\]
The following discrete-time Markov additive renewal theorem is lifted from  Proposition 9.3 in \cite{JM}. 
\begin{theorem}[Markov Additive Renewal Theorem]\label{prop}Given a sequence of functions $g_1,g_2,\dots,g_{N}$ that are directly Riemann integrable, we have, for $j\in E$,
\begin{equation}
\lim_{t\to \infty}\int_{\mathbb{R}}g_j (t-s)R_{i,j}({\rm d}s)=\frac{\pi_j \int_{0}^{\infty}g_{j}(y){\rm d}y}{\sum_{j=1}^N\pi_j\eta_j},
\end{equation}
as soon as $\sum_{j=1}^N\pi_j\eta_j\in(0,\infty)$, 
where $\pi_i$ is the stationary distribution for the chain $M$.
\end{theorem}

\section*{Acknowledgements}

AEK  acknowledges support from EPSRC grant number EP/L002442/1. AEK and VR acknowledge support from EPSRC grant number EP/M001784/1. VR acknowledges support from CONACyT grant 234644. This work was undertaken whilst VR was on sabbatical at the University of Bath, he gratefully acknowledges the kind hospitality of the Department and University.

\end{document}